\documentclass[12pt,reqno]{amsart}

\usepackage{amssymb,bm,mathrsfs}

\usepackage{enumerate}
\usepackage{tikz}
\usepackage[centering,width=6in]{geometry}
\usepackage[colorlinks,pdfstartview=FitH]{hyperref}

\hypersetup{
    colorlinks,
    citecolor=blue,
    filecolor=black,
    linkcolor=red,
    urlcolor=black
}

\theoremstyle{plain}
\newtheorem{theorem}{Theorem}[section]
\newtheorem{proposition}[theorem]{Proposition}
\newtheorem{lemma}[theorem]{Lemma}

\newtheorem{op}[theorem]{Open Problem}

\theoremstyle{definition}
\newtheorem{definition}[theorem]{Definition}
\newtheorem{example}[theorem]{Example}

\theoremstyle{remark}
\newtheorem{rem}[theorem]{Remark}
\numberwithin{equation}{section}

\DeclareMathOperator{\Var}{Var}
\DeclareMathOperator{\hdim}{dim_H}
\DeclareMathOperator{\Fdim}{dim_F}

\newcommand{\N}{\mathbb{N}}
\newcommand{\Z}{\mathbb{Z}}
\newcommand{\R}{\mathbb{R}}

\newcommand{\E}{\mathbb{E}}

\DeclareFontFamily{U}{mathx}{}
\DeclareFontShape{U}{mathx}{m}{n}{<-> mathx10}{}
\DeclareSymbolFont{mathx}{U}{mathx}{m}{n}
\DeclareMathAccent{\widehat}{0}{mathx}{"70}
\DeclareMathAccent{\widecheck}{0}{mathx}{"71}

\usepackage{graphicx} % Required for inserting images

\title[Fourier dimension of the graph of fBm with $H\ge 1/2$]{Fourier dimension of the graph of fractional Brownian motion with $H\ge 1/2$}
\author{Chun-Kit Lai}
\address{San Francisco State University Department of Mathematics, 1600 Holloway Ave, San Francisco, CA 94132
}
\curraddr{}
\email{cklai@sfsu.edu}
\thanks{}

\author{Cheuk Yin Lee}
\address{School of Science and Engineering, The Chinese University of Hong Kong, Shenzhen, Guangdong 518172, China}
\email{leecheukyin@cuhk.edu.cn}

\begin{document}



\begin{abstract}
    %We show that the Fourier dimension of the graph of fractional Brownian motion with Hurst index $H\in[\frac12,1)$ as well as the symmetric $\alpha$-stable process with $\alpha\in[1,2]$ are almost surely equal to 1. In the proof, we introduce a combinatorial integration by parts formula to compute the iterated integrals of the $2q$-th moments. We also establish several variance and derivative of variance estimates of the fractional Brownian motion and combine it with the integration by part formula using a multivariate chain rule known as the Fa\`a di Bruno's formula. 
    %The result of Fraser, Opronen and Sahlsten (2014/) showed that the graph of fractional Brownian motion is almost surely not a Salem set, answering a question of Kahane (1993). In this paper, w
    We prove that the Fourier dimension of the graph of fractional Brownian motion with Hurst index greater than $1/2$ is almost surely 1. This extends the result of Fraser and Sahlsten (2018) for the Brownian motion and confirms part of the conjecture of Fraser, Orponen and Sahlsten (2014). We introduce a combinatorial integration by parts formula to compute the moments of the Fourier transform of the graph measure. The proof of our main result is based on this integration by parts formula together with Fa\`a di Bruno's formula and strong local nondeterminism of fractional Brownian motion. We also show that the graph of a symmetric $\alpha$-stable process has Fourier dimension 1 almost surely when $\alpha \in [1,2]$ and is a Salem set when $\alpha = 1$.
\end{abstract}

\maketitle
%\tableofcontents

\section{Introduction}

The study of fractal properties of the graph of stochastic processes has been a significant part of the theory in fractal geometry and probability. This was first initiated in 1953 by Taylor \cite{Taylor53} who showed that the Hausdorff dimension of the graph of the one-dimensional standard Brownian motion is almost surely $3/2$.  This result was generalized to the fractional Brownian motion by Orey \cite{Orey1970}, who showed that the Hausdorff dimension of the graph of fractional Brownian motion with Hurst index $H\in(0,1)$ is almost surely $2-H$. 
Recall that a \emph{fractional Brownian motion} (fBm) with Hurst index $H$ is a continuous Gaussian process $B = \{B(t) : t \ge 0\}$ with mean 0 and covariance
\begin{align}\label{E:fBm:cov}
   \mbox{Cov} (B(s), B(t)) = \E[B(t)B(s)] = \frac12 (t^{2H}+s^{2H}-|t-s|^{2H}) \quad \text{for $t, s \ge 0$}.
\end{align}
Note that the fBm with $H = 1/2$ is the standard Brownian motion.

Another generalization is the symmetric $\alpha$-stable process with $\alpha\in(0,2]$, which is a self-similar L\'{e}vy process whose sample paths are not continuous unless $\alpha = 2$.
%, which is the Brownian motion. 
Recall that $X=\{X(t) : t \ge 0\}$ is a \emph{L\'{e}vy process} if $X$ has independent and stationary increments; $X$ is called a \emph{symmetric $\alpha$-stable process} if it is a L\'{e}vy process whose characteristic function is
\begin{equation}\label{eq_stable-process}
\E \left[ e^{i \xi (X(t)-X(s))}\right] = e^{- (t-s)|\xi|^{\alpha} } \quad \text{for $t\ge s>0$ and $\xi\in \R$.}
\end{equation}
When $\alpha=2$, $X$ is the standard Brownian motion.
The graph of the symmetric $\alpha$-stable process has Hausdorff dimension $\max\{1, 2-1/\alpha\}$ almost surely. We refer to Falconer's book \cite[Chapter 16]{Falconer-book} for an excellent exposition about these basic results. 

The Hausdorff dimension results have also been generalized to Gaussian random fields, namely, stochastic processes $\{X(t): t \in \R^N\}$ indexed by $t \in \R^N$ with values $X(t)\in \R^d$, whose finite-dimensional distributions are all Gaussian.
This includes, for example, the work of Adler \cite{Adler1977} for Gaussian random fields with stationary increments, and Xiao \cite{X09} for Gaussian random fields under more general conditions.
For related developments on fractal properties of Gaussian random fields, the reader is referred to the survey article by Xiao \cite{X13}.

Apart from the Hausdorff dimension, there is a growing interest in the Fourier dimension, which is another notion of dimension that has many intriguing properties. In this paper,  the Fourier transform of a finite Borel measure $\mu$ on $\R^d$ is defined by 
\[
\widehat{\mu}(\xi)  = \int_{\R^d} e^{-2\pi i \xi\cdot x}d\mu(x).
\]
The Fourier dimension of a Borel probability measure $\mu$ is defined by 
\[
\Fdim \mu= \sup \{\beta \ge 0:   \ |\widehat{\mu}(\xi)|\lesssim |\xi|^{-\beta/2}\}.
\]
The \emph{Fourier dimension} of a  Borel set $E$ in $\R^n$, denoted by $\Fdim E$, is 
\[
\Fdim E = \sup \{\beta \in [0,n]: \exists \mu \ \text{supported on $E$ such that}  \ |\widehat{\mu}(\xi)|\lesssim |\xi|^{-\beta/2}\}.
\] 
Throughout, the notation $f(x)\lesssim g(x)$ means that there exists a constant $0<C<\infty$ such that $f(x)\le Cg(x)$ for all $x$.  
The study of the Fourier dimension lies at the interface of classical harmonic analysis, geometric measure theory, number theory, and probability. See, e.g., Mattila's book \cite{Mattila} for different facets about the interplay between Fourier decay and sets of fractional dimensions. 

Fourier dimension encapsulates many arithmetic properties of sets. By a classical result of Davenport, Erd\H{o}s and LeVeque \cite{DEL63}, if a compact set $E\subset \R$ supports a measure $\mu$ with $\Fdim\mu>0$, then $\mu$-almost all $x\in E$ are normal to all bases $b$ when expanded in any $b$-digit expansions. {\L}aba and Pramanik \cite{LP09} showed that, under certain technical assumptions, a set with sufficiently large Fourier dimensions contains a non-trivial arithmetic progression of three terms. Despite this property being false in general \cite{S2017}, the paper initiated a series of new research concerning the connection between the arithmetic patterns and Fourier decays (see, e.g., \cite{LP22} and the references therein). In classical harmonic analysis, Fourier decay due to the curvature of the sphere plays an important role in the Fourier restriction phenomenon on the sphere. The related Fourier restriction conjecture has become one of the central conjectures in harmonic analysis. Indeed, a measure with a positive Fourier dimension is already enough for Fourier restriction. We refer to \cite[Chapter 19-22]{Mattila} for more details. A measure supported on a hyperplane does not have Fourier decay or Fourier restriction. Therefore, having a positive Fourier dimension is considered to be some kind of positive curvature conditions. 

Because of its importance in many fields, computation of Fourier dimension, for both deterministic and random sets, has also become a very natural problem for researchers. It is well known that
$$
\Fdim E \le \hdim E.
$$
A set whose Fourier and Hausdorff dimensions are equal is called a \emph{Salem set}. It is never easy to find explicit deterministic Salem sets. Indeed, the middle-third Cantor set has Fourier dimension zero. One can prove this by a direct calculation, or appeal to the normal number results we mentioned earlier since no points on the middle-third Cantor sets can be normal to the basis 3. In the current literature, the only known deterministic Salem sets are all constructed via diophantine approximation which was first discovered by Kaufman \cite{Kaufman81}. R. Fraser and Hambrook recently generalized Kaufman's result to any dimensions \cite{FH2023}. Many measures are also determined to have positive Fourier dimensions, though it is unknown if their supporting sets are Salem. Those include the Patterson--Sullivan measures on Schottky groups \cite{BD17}, non-linear image of self-similar measures \cite{MS2018} and  badly approximable numbers \cite{Kaufman80} (see also \cite{JS16} for some generalization). 

Salem sets or sets with positive Fourier dimensions are  ubiquitous in random settings. Kahane \cite{Kah} showed that the image of a set under the fractional Brownian motion is almost surely Salem. 
Kahane's result was generalized to Gaussian random fields with stationary increments \cite{SH06}, Brownian sheet \cite{KWX06}, and fractional Brownian sheets \cite{WX2007}. Moreover, random Cantor sets obtained by randomly selecting basic intervals with spatial independence are in general Salem sets (see, e.g., \cite{S2017,SS2018}). 

In \cite{SH06} and \cite{K1993}, it was raised as an open question whether the graph and the level sets of fractional Brownian motion (fBm) are Salem sets. Mukeru \cite{M2018} showed that the zero sets of fBm is almost surely Salem by studying the Fourier transform of its local time. 
For the graph of functions, Fraser, Orponen and Sahlsten \cite{FOS2014} proved that the Fourier dimension of the graph of any continuous function is at most 1, which implies that the graph of fBm is almost surely not Salem.
They conjectured that the Fourier dimension of the graph of fBm is equal to 1.
While this remained an open problem for fBm, Fraser and Sahlsten \cite{FS2018} later proved that the Fourier dimension of the graph of the standard Brownian motion is exactly 1, the maximal value possible.

\subsection{Main results and contributions.} The main result of this paper is to show that the Fourier dimension is also equal to 1 for fractional Brownian motion with Hurst index $H\ge 1/2$.

\begin{theorem}\label{th:dimF:fBm}
    Let ${\mathcal B} = \{ B(t): t \ge 0\}$ be a fractional Brownian motion with Hurst index $H\in [1/2,1)$.
    Then the Fourier dimension of the graph of ${\mathcal B}$ is almost surely 1. 
\end{theorem}

The proof of this theorem is based on the estimation of the $2q$-th moments of the two-dimensional graph measure induced by the standard Lebesgue measure on the time interval. The estimate can be divided into vertical and horizontal directions (see Proposition \ref{prop_sufficiency} for a precise statement). In the vertical direction, it is dominated by the image measure of the process, which has been well studied in the literature. The challenge is to obtain the estimate for the horizontal direction, in which we expect the behavior is similar to the Lebesgue measure.

In the horizontal direction for the standard Brownian motion, Fraser and Sahlsten \cite{FS2018} used It\^{o}'s formula to transfer the required integral into an It\^{o} stochastic integral. In this proof, the key ingredient is that  the quadratic variation measure is exactly the Lebesgue measure, which allowed  them to represent the graph measure essentially as an It\^{o} integral. Since an It\^{o} integral is a martingale, the Burkholder--Davis--Gundy inequality can be applied to obtain the required moment estimates. 
% Their approach is elegant, but is unfortunately not effective for other processes. 
In \cite{DL2023}, a similar approach was used to study the Fourier dimension for some time-changed Brownian motion with strictly increasing time $V(t)$. Among many other results,  Dysthe and the first named author showed that the Fourier dimension is at least $2/3$ if $V(t)$ is bi-Lipschitz, but is obviously far from sharp.  We note that stochastic calculus has also been established for fBm; see, e.g., \cite{Nualart}, \cite{Mishura} and \cite{fBm-integral}.
% the above approach has many technical barriers and does not seem to carry over to fBm.
However, the stochastic calculus approach used in \cite{FS2018, DL2023} does not seem to carry over directly to fBm, since the computations and moment estimates do not simplify the same way as in the case of Brownian motion.

\subsection{Overview and novelty of the proof} 
To prove Theorem \ref{th:dimF:fBm}, we will develop a 
%completely new 
different approach to the horizontal direction through three combinatorial steps, the last two of which are believed to be novel. 
The following is the outline of the proof and the organization of the paper. 

In Section \ref{Pre}, we will lay out the basic strategy to estimate the $2q$-th moments. In particular, the first combinatorial step, which is a well-known step due to Kahane \cite{Kahane64}, will be presented. We will decompose the $2q$-th moment of the graph measures into $\binom{2q}{q}$ many iterated integrals of $2q$ variables $u_1,\dots, u_{2q}$ whose domain is the simplex where $0\le u_1<\cdots<u_{2q}\le 1$ (see Lemma \ref{eq_expect_simplex}). 

Using the fact that the fBm is a self-similar Gaussian process,  each iterated integral is a Fourier transform of the characteristic function of a linear combination of fBm on the simplex, which is of the form 
\begin{align}\label{E:exp(-var)}
    \exp\left\{-\Var\left(\sum_{i=1}^{2q}\varepsilon_i B(u_i)\right)\right\}
\end{align}
where $\Var(X)$ denotes the variance of a random variable $X$.  From classical Fourier analysis, we know that Fourier decay can be achieved by performing integration by parts. 
In order to obtain the correct order of decay, we need to apply integration by parts $q$ times on these $2q$ iterated integrals. 
In Section \ref{sec:bypart}, we present our second combinatorial step, which is to devise a systematic way to perform integration by parts multiple times, resulting in a combinatorial formula.
We describe all the combinatorial terms that appeared in these $q$ successive integrations by parts (see Proposition \ref{prop_by_part-iterate}). In particular, the partial derivatives of the function \eqref{E:exp(-var)} appear.

Indeed, the combinatorial integration by parts formula makes no reference to the process. For the symmetric $\alpha$-stable process, we can readily estimate every term using independence and stationarity of the increments and the fundamental theorem of calculus. In Section \ref{sec: alpha-stable}, we prove the following result:

\begin{theorem}\label{th:dimF:stable}
    Let ${\mathcal X} = \{ X(t): t \ge 0\}$ be a symmetric $\alpha$-stable process with $\alpha\in(0,2]$.
    Then the Fourier dimension of the graph of ${\mathcal X}$ is at least $\min\{1,\alpha\}$ almost surely. Therefore, the graph of ${\mathcal X}$ has Fourier dimension 1 almost surely when $\alpha \ge 1$,
    and is a Salem set when $\alpha = 1$.
\end{theorem}

 When $\alpha=2$, ${\mathcal X}$ is the standard Brownian motion and Theorem \ref{th:dimF:stable} recovers the main result of Fraser and Sahlsten \cite{FS2018}. In Section \ref{section 5}, we begin our setup for the proof of Theorem \ref{th:dimF:fBm} under a similar strategy as Theorem \ref{th:dimF:stable}. Finally, it all boils down to an estimate for the $2q$-iterated integrals of the mixed partial derivatives of \eqref{E:exp(-var)} (see Theorem \ref{theorem_hard-estimate-introd}). We will first prove Theorem \ref{th:dimF:fBm} assuming Theorem \ref{theorem_hard-estimate-introd}.

 The remainder of the paper will be devoted to proving Theorem \ref{theorem_hard-estimate-introd}. As is well known, fBm does not have independent increments when $H\ne 1/2$, so it is more challenging to handle integrals involving \eqref{E:exp(-var)} and its mixed partial derivatives. 
 In Section \ref{Section 6}, we will develop a combinatorial setup for the estimation where our third combinatorial step is presented.  
 In order to estimate the partial derivatives of \eqref{E:exp(-var)} which is a function of the form $\exp(g({\bf x}))$, we will use a general chain rule for high-order derivatives known as {\it Fa\`{a} di Bruno's formula} \cite{FdB1855, FdB1857}.
 In particular, we will use a multi-variate, combinatorial version of Fa\`{a} di Bruno's formula to obtain a combinatorial derivative formula for $\exp(g({\bf x}))$ (see Lemma \ref{lem:fdb} below).
 Because of the covariance structure of the fBm, all mixed partial derivatives of order three or higher vanish and hence only the first and the second derivatives appear in this combinatorial formula.

In Section \ref{sec:variance}, we will establish estimates for the variance in \eqref{E:exp(-var)} and its first and second derivatives. The variance estimates are established using the \emph{strong local nondeterminism} property of fBm.
While the variance estimates are valid for any $H\in(0,1)$, the first derivative estimates rely strongly on the fact that $H\ge 1/2$ as the inequalities will change sign if we go below $1/2$.
The notion of local nondeterminism (LND) was first introduced by Berman \cite{B73} for Gaussian processes to overcome difficulties due to the lack of independent increments.
Pitt \cite{P78} extended the definition of LND to Gaussian random fields and proved a form of LND for fBm which was later called strong LND \cite{CD82, MP87}.
Strong LND has been further developed and applied to study various properties of Gaussian random fields. We refer to the article of Xiao \cite{X08} and the references therein for typical applications of strong LND and \cite{L22, LX23, LX25, KL25} 
for some recent developments.

Finally, in Section \ref{S:Fourier:decay}, we gather all the estimates together to obtain the upper bound via the Fa\`{a} di Bruno formula, which will complete the proof of Theorem \ref{theorem_hard-estimate-introd} and hence Theorem \ref{th:dimF:fBm}. 

   \subsection{Open problems.}  The following two problems remain open: 

   \begin{op}\label{op1}
   Determine whether the Fourier dimension of the graph of a fractional Brownian motion is $1$ when $H\in (0,1/2)$.
   \end{op}

   \begin{op}\label{op2}
Determine whether the Fourier dimension of the graph of a symmetric $\alpha$-stable process  is $1$ when $\alpha \in (0,1)$.
\end{op}
   
  We notice that the proof of Theorem \ref{th:dimF:fBm} relies on the local integrability of $x^{2H-2}$ around $x = 0$.  Moreover, the first derivative estimates in Proposition \ref{lem:dg:bd} will change sign when $H<1/2$. A much more refined estimate is needed to study Problem \ref{op1} and it is unclear even if Fourier dimension $1$ is still achievable.

  For the $\alpha$-stable process, we will see in the proof that the decay in the horizontal direction can achieve the desired decay rate even $\alpha<1$. Unfortunately, the decay in the vertical direction, which follows from the decay of the image measures, is at most $\alpha$ (See Proposition \ref{prop-vertical-stable}). When $\alpha<1$, the vertical direction does not give enough decay rate to guarantee that its Fourier dimension on the graph is 1, leaving Problem \ref{op2} open. Despite these extra technicalities, we believe that the integration by parts setup has laid the foundation to tackle these problems and to study Fourier decays for other stochastic processes and random fields in the future.

\noindent{\it Note: In October 2025, the second-named author and Samy Tindel built on the setting of this paper and resolved Problem \ref{op1}, showing that the Fourier dimension of the graph of fBm with $H<1/2$ is almost surely 1 \cite{lee2025fourierdimensionfractionalbrownian}. }

%Our method also works for the time-changed Brownian motion $B(V(t))$ in \cite{DL2023} and we need a variant of the integration by part formula to add in the consideration of $V'(t)$. If $V'$ is uniformly bounded away from zero, we will obtain the graph has Fourier dimension 1. The case will be more difficult if $V$ has multifractal behavior.  We anticipate such a study in the future. 

\section{Preliminaries}\label{Pre}
We begin by collecting some results that are valid for general processes in the next two sections. %The precise definitions of symmetric $\alpha$-stable process and fractional Brownian motion will be given in Section \ref{sec: alpha-stable} and Section \ref{section 5}.
\subsection{General method for finding Fourier dimension.} 
A classical way of finding the Fourier dimension of a random measure is to compute all the $2q$-th moments of the measure. This method was first used by Kahane \cite{Kah}. General statements for this approach and special cases of the following lemma can be found in \cite{Ek2016} and \cite{DL2023}. We are going to use the following recent version due to Fraser.  

\begin{lemma}\cite[Lemma 8.1]{Jon2024}\label{randmeaslem}
Let $\mu = \mu_\omega$ be a random measure on $\R^d$ that is compactly supported almost surely. Suppose further that there exists $C_q < \infty$ such that 
\begin{equation}\label{exptdecayassump}
\mathbb{E}\left[|\widehat{\mu}(\xi)|^{2q}\right] \le C_q |\xi|^{-\gamma q} \quad \text{for all $\xi \in \R^d$ and $q\in\N$.} 
\end{equation}
Then, almost surely,
\begin{equation}\label{ftbound}
\Fdim \mu \ge \gamma. 
\end{equation}
\end{lemma}

%All previous formulations require the constant to satisfy the condition $C_q \lesssim q^{cq}$ for some $c>0$. Recently,  Fraser \cite{Jon2024} introduced the idea of the Fourier spectrum that interpolates the Fourier dimension and Sobolev dimension. It was shown that if $\widehat{\mu}$ is Lipschitz continuous (which is the case if $\mu$ is compactly supported), then the Fourier spectrum is continuous up to the endpoints, which is exactly the Fourier dimension. This removes the requirement about the growth rate of the constant $C_q$. 
Note that, in the above formulation, there is no condition on the growth rate of $C_q$ with $q$, which is the novel feature of \cite{Jon2024}.
If \eqref{exptdecayassump} holds with a specific growth rate, say $C_q \lesssim q^{cq}$, then a more precise almost sure bound can be derived for the decay of $|\widehat{\mu}(\xi)|$; see, e.g., \cite{Kah, KWX06, SH06}.

There are other large deviation techniques in computing Fourier dimensions of random Cantor sets; see, e.g., \cite{S2017,SS2018}. Recently, the technique was extended to Mandelbrot cascade measures \cite{Mandelbrot-harmonic, Fdim-mandelbro}.

\subsection{Push-forward measures on graphs.} Let $X = \{X(t) : t \ge 0\}$ be a real-valued stochastic process.  For ${\bf t} = (t_1,\dots, t_q)$, ${\bf s} = (s_1,\dots, s_q)\in{\mathbb R}^q$, we define
\begin{align}\label{E:G(xi,t,s)}
G(\xi,{\bf t},{\bf s}) = {\mathbb  E} \left[e^{-2\pi i \xi \sum_{j=1}^q(X(t_j)-X(s_j))}\right].
\end{align}
 The graph of $X$ over the interval $[0,1]$ is the random set
\[
    \mathcal{G}(X) = \left\{ (t, X(t)) : t \in [0,1] \right\}.
\]
A natural Borel measure $\mu_{\mathcal{G}}$ can be defined on the graph of $X$ via the push-forward of the Lebesgue measure (denoted by $m$):
\[
    \mu_{\mathcal{G}}(E) = m\left\{ t \in [0,1]  : (t, X(t)) \in E\right\}.
\]
Its Fourier transform is given by 
\[
   \widehat{\mu_{\mathcal{G}}}(\xi_1,\xi_2) = \int_0^1 e^{-2\pi i (\xi_1 t+ \xi_2 X(t))}~dt.
\]
As Lemma \ref{randmeaslem} suggests, we may compute the almost sure Fourier dimension of the graph by estimating all the $2q$-th moments of the random measure   $\widehat{\mu_{\mathcal{G}}}(\xi_1,\xi_2)$. 
In doing so, we may rewrite $|\widehat{\mu}_{{\mathcal{G}}}(\xi_1,\xi_2)|^{2q}$ as the $2q$ iterated integral
\[
|\widehat{\mu}_{{\mathcal{G}}}(\xi_1,\xi_2)|^{2q} = \int_{[0,1]^{2q}} e^{-2\pi i\xi_1\left(\sum_{i=1}^{q}(t_{i}-s_{i})\right)} \cdot e^{-2\pi i\xi_2\left(\sum_{i=1}^{q}(X(t_{i})-X(s_{i}))\right)}d{\bf t}~d{\bf s}
\]
where $d{\bf t} = dt_1\cdots d t_q$ and $d{\bf s} = ds_1\cdots ds_q$. Taking expectation and using \eqref{E:G(xi,t,s)},
%noticing that $\sum_{i=1}^{q}(B^H(t_{i})-B^H(s_i))$ is a Gaussian random variable,
we have 
\begin{equation}\label{eq_expect}
\E\left[|\widehat{\mu}_{{\mathcal{G}}}(\xi_1,\xi_2)|^{2q}\right] = \int_{[0,1]^{2q}} e^{-2\pi i\xi_1\left(\sum_{i=1}^{q}(t_{i}-s_i)\right)} \cdot G(\xi_2,{\bf t},{\bf s})~d{\bf t}~d{\bf s}.
\end{equation}

%The following lemma is immediate.
\begin{lemma}\label{lemma_bounding_by_image}
   If $G$ is a non-negative real-valued function, then
    \[
\E\left[|\widehat{\mu_{\mathcal G}}(\xi_1,\xi_2)|^{2q}\right] \le  \E\left[|\widehat{\nu}(\xi_2)|^{2q}\right] \quad \text{for any $\xi_2 \in \R$,}
\]
where $\nu$ is the image measure of $X$, i.e.,
$\nu(E) = m\left\{t\in[0,1]: X(t)\in E\right\}$.
\end{lemma}

\begin{proof}
    Notice that $\widehat{\nu}(\xi) = \int_0^1 e^{-2\pi i \xi X(t)}dt$. Applying absolute value to the integrand in (\ref{eq_expect}) and using the fact that $G$ is a non-negative real-valued function, we have
    $$
    \E\left[|\widehat{\mu}_{{\mathcal{G}}}(\xi_1,\xi_2)|^{2q}\right] \le \int_{[0,1]^{2q}} G(\xi_2,{\bf t},{\bf s})~d{\bf t}~d{\bf s}.
    $$
    But the iterated integral on the right-hand side is exactly $\E[|\widehat{\nu}(\xi_2)|^{2q}]$. This completes the proof. 
\end{proof}

\begin{proposition}\label{prop_sufficiency}
    Suppose that we have:
\begin{enumerate}
    \item[{\rm (1)}] {\bf (Vertical bound)} $\E\left[|\widehat{\nu}(\xi_2)|^{2q}\right]\lesssim |\xi_2|^{-\gamma_2\cdot q}$, $\forall \xi_2\ne 0$ and $\forall q\in{\mathbb N}$; and
    \item[{\rm (2)}] {\bf (Horizontal bound)} $\E\left[|\widehat{\mu_{\mathcal G}}(\xi_1,\xi_2)|^{2q}\right] \lesssim |\xi_1|^{-\gamma_1\cdot  q}$, $\forall \xi_1\ne 0$ and $\forall q\in{\mathbb N}$,
\end{enumerate}
where all the implicit constants depend only on $q$.  Then $\Fdim {\mathcal G}(X) \ge \min \{\gamma_1,\gamma_2\}$. 
\end{proposition}

\begin{proof}
We first notice that 
\begin{align}\label{E:E:mu:hat:bounded}
    \E\left[|\widehat{\mu_{\mathcal G}}(\xi_1,\xi_2)|^{2q}\right] \le 1 \quad \text{for all $\xi \in \R^2$,}
\end{align}
so it suffices to establish \eqref{exptdecayassump} for $|\xi|>1$.
Let us write $\xi =(\xi_1,\xi_2)$ in polar coordinates so that $\xi_1 =  |\xi|\cos \theta_{\xi}$ and $\xi_2 =  |\xi|\sin \theta_{\xi}$, with $\theta_\xi\in[-\pi,\pi)$. Then we decompose $\R^2\setminus\{0\} = {\mathcal H}\cup {\mathcal V}$ into two regions where 
$$
{\mathcal H}: = \left\{(\xi_1,\xi_2): \theta_{\xi} \in \left[-\tfrac{\pi}{4}, \tfrac{\pi}{4}\right]\cup \left[\tfrac{3\pi}{4},\pi\right)\cup \left[-\pi,-\tfrac{3\pi}{4}\right]\right\}$$
and 
$$
{\mathcal V}: =\left\{(\xi_1,\xi_2): \theta_{\xi} \in \left[\tfrac{\pi}{4},\tfrac{3\pi}{4}\right)\cup \left[-\tfrac{3\pi}{4},-\tfrac{\pi}{4}\right)\right\}.
$$
For all $\xi\in {\mathcal H}$, writing $\xi_1 = |\xi|\cos\theta_{\xi}$, we have $|\cos\theta_{\xi}|\ge 1/\sqrt{2}$, so the horizontal bound (2) implies that 
$$
\E\left[|\widehat{\mu_{\mathcal G}}(\xi_1,\xi_2)|^{2q}\right] \lesssim |\xi_1|^{-\gamma_1\cdot  q} \le \left(\sqrt{2}\right)^{\gamma_1 q}\cdot |\xi|^{-\gamma_1\cdot  q}. 
$$
By a similar argument, for all $\xi\in {\mathcal V}$, writing $\xi_2 = |\xi|\sin\theta_{\xi}$ with $|\sin\theta_{\xi}|\ge 1/\sqrt{2}$, we have, by the vertical bound (1) and Lemma \ref{lemma_bounding_by_image}, that 
$$
\E\left[|\widehat{\mu_{\mathcal G}}(\xi_1,\xi_2)|^{2q}\right] \lesssim \left(\sqrt{2}\right)^{\gamma_2 q}\cdot |\xi|^{-\gamma_2\cdot  q}. 
$$
Combining both cases, together with \eqref{E:E:mu:hat:bounded}, we deduce that
$$
\E\left[|\widehat{\mu_{\mathcal G}}(\xi_1,\xi_2)|^{2q}\right] \lesssim |\xi|^{-\gamma q},
$$
where $\gamma = \min\{\gamma_1,\gamma_2\}$.
Hence, the desired result follows from Lemma \ref{randmeaslem}. 
\end{proof}
Proposition \ref{prop_sufficiency} now suggests that we need to bound all the $2q$-th moments of $\widehat{\nu}$ and $\widehat{\mu_{\mathcal G}}$ in order to compute or find a lower bound for the Fourier dimension of $\mathcal{G}(X)$. The vertical bound in Proposition \ref{prop_sufficiency} has been known for many stochastic processes. The most challenging part is to obtain the horizontal bound in Proposition \ref{prop_sufficiency}. We  will first use a similar simplex decomposition method due to Kahane \cite{Kah},  then we introduce a combinatorial integration by parts formula in the next section to facilitate the estimation.

\subsection{Simplicial decomposition.}  
As observed by Kahane, the integral in (\ref{eq_expect}) is invariant under permutation of all the coordinates of ${\bf t}$ or all the coordinates of ${\bf s}$. Thanks to this property, we may decompose $[0,1]^{2q}$ into $(2q)!$ many simplicial regions to see that
\[
\E\left[|\widehat{\mu}(\xi_1,\xi_2)|^{2q}\right]  = (q!)^2\int_{{\mathsf S}} e^{-2\pi i \xi_1 \left(\sum_{i=1}^{q}(t_{i}-s_{i})\right)}\cdot G(\xi_2, {\bf t},{\bf s})~d{\bf t}~d{\bf s}
\]
where ${\mathsf S}$ is a union of simplexes on $[0,1]^{2q}$ with the following representation:
\[
{\mathsf S} = \{(t_1,\dots t_q,s_1,\dots, s_q)\in[0,1]^{2q}: 0\le t_1<\cdots<t_q\le 1 \ \mbox{and} \ 0\le s_1<\cdots<s_q\le 1 \}.
\]
Now, the preceding integral can be further decomposed by introducing the sequences 
$\boldsymbol{\varepsilon}=(\varepsilon_j)_{j=1}^{2q}\in\{-1,1\}^{2q}$ such that $\sum_{j=1}^{2q}\varepsilon_j = 0$. Denote by ${\mathcal A}_{2q}$ the collection of such sequences:
\begin{align}\label{A_{2q}}
    {\mathcal A}_{2q} = \Big\{ \boldsymbol{\varepsilon}\in\{-1,1\}^{2q} : \sum_{j=1}^{2q}\varepsilon_j = 0 \Big\}.
\end{align}
There are $\binom{2q}{q}$ such sequences in ${\mathcal A}_{2q}$. Rearranging $t_1<t_2<\cdots <t_q$ and $s_1<s_2<\cdots <s_q$ into $u_1<u_2<\cdots< u_{2q}\le 1$, we obtain the following formula:
\begin{lemma}\label{eq_expect_simplex}
\[
\E\left[|\widehat{\mu}(\xi_1,\xi_2)|^{2q}\right] = (q!)^2\sum_{\boldsymbol{\varepsilon}\in{\mathcal A}_{2q}} \int_{0\le u_1<\cdots <u_{2q}\le 1} e^{-2\pi i \xi_1 \left(\sum_{j=1}^{2q}\varepsilon_j u_j\right)}\cdot G_{\boldsymbol{\varepsilon}}(\xi_2, {\bf u})~d{\bf u},
\]
where $G_{\boldsymbol{\varepsilon}}(\xi_2, {\bf u})
    = \E\big[ e^{-2\pi i \xi_2 \sum_{j=1}^{2q} \varepsilon_j X(u_j)}\big]$.
\end{lemma}
% \textcolor{magenta}{What do you think about Comment 12 of report X?}

%Note that the iterated integral is a Fourier transform of $G$ over the simplicial region and $\xi_1 \ne 0$ when dealing with the horizontal bound.  

\section{A combinatorial integration by parts formula}\label{sec:bypart}

 %In estimating all the $2q$-th moments, it is natural to unwrap the integral into $2q$ iterated integrals as in (\ref{eq_expect}). 
 %Note that $G$ is a differentiable function. 
 From classical Fourier analysis, we know that Fourier decay of a sufficiently smooth function can be derived by performing successive integrations by parts. The goal of this section is to introduce a systematic way to carry out integrations by parts for the $2q$ iterated integrals in Lemma \ref{eq_expect_simplex}.

\subsection{Integration by parts.} 
As we will see later, thanks to scaling property, it suffices to consider the following iterated integral:
$$
\int_{0\le u_1<\cdots <u_{2q}\le T} e^{-2\pi i \lambda \left(\sum_{j=1}^{2q}\varepsilon_j u_j\right)}\cdot G( {\bf u})~d{\bf u}
$$
where $G$ is an infinitely differentiable function over the simplex as defined in the domain of integration for $T>0$.

Let us introduce some necessary notations to describe the integration by parts process. The symbol $\ast$ will be used to indicate the removal of a particular coordinate variable. 
The values of the entries of $\boldsymbol\varepsilon$ will also be extended from $\{-1,1\}$ to $\Z \cup \{\ast\}$.
For each $\boldsymbol\varepsilon = (\varepsilon_1,\dots, \varepsilon_n)\in (\Z\cup \{\ast\})^n$ with $\{j: \varepsilon_j = \ast\} = \{j_1,\dots, j_k\}$, let 
$$
\Delta_{\boldsymbol\varepsilon, T} = \{(u_1,\dots, u_{j_1-1},u_{j_1+1},\dots, u_{j_k-1},u_{j_k+1},\dots, u_n): 0\le u_1<\cdots< u_n\le T\}
$$
denote the simplex in $\R^{n-k}$ with the coordinates $u_{j_1},\dots, u_{j_k}$ removed.  For a smooth non-negative function $G$ on $\R^{n-k}$, let
\begin{equation}\label{eq_I[e,G]}
{\mathcal I}[\boldsymbol\varepsilon, G] = {\mathcal I}_T^\lambda[\boldsymbol\varepsilon, G] := \int_{\Delta_{\boldsymbol\varepsilon, T}} e^{-2\pi i \lambda \langle\boldsymbol\varepsilon, {\bf u}\rangle}\cdot G( {\bf u})~d{\bf u}
\end{equation}
denote the $n-k$ iterated integral over $\Delta_{\boldsymbol{\varepsilon},T}$
where $\langle\boldsymbol\varepsilon, {\bf u}\rangle$ is identified as an $\R^{n-k}$ inner product with the position $\ast$ removed. For example, 
$$
{\mathcal I}[\ast, \varepsilon_2,\ast, \varepsilon_4, G] = \iint_{0\le u_2<u_4\le T} e^{-2\pi i \lambda (\varepsilon_2u_2+\varepsilon_4u_4)} G(u_2,u_4) ~du_2du_4. 
$$
We remind the readers that ${\mathcal I}[\boldsymbol\varepsilon, G]$ is a function on $\lambda$ and $T$. To simplify notations in the proof,  we will often write ${\mathcal I}[\boldsymbol\varepsilon, G]$ instead of ${\mathcal I}^\lambda_T[\boldsymbol\varepsilon, G]$. % but will write ${\mathcal I}^\lambda_T[\boldsymbol\varepsilon, G]$ for instances where we want to emphasize its dependence on $\lambda$ and $T$.

%Let $\boldsymbol\varepsilon\in (\Z\cup\{\ast\})^n$ and $j \in \{1,\dots, n\}$.
%$\varepsilon_j\ne 0$ where $j \in\{1,\dots, n\}$ 
Suppose that $\varepsilon_{j-1}, \varepsilon_{j+1}\ne \ast$ if $j\not\in\{1,n\}$; $\varepsilon_2\ne\ast$ if $j=1$;  and $\varepsilon_{n-1}\ne\ast$ if $j = n$.
Define
$$\Phi_j^-(\boldsymbol\varepsilon) = \left\{\begin{array}{cc}
    (\varepsilon_1,\cdots, \varepsilon_{j-1}+\varepsilon_j, \underbrace{\ast}_{j\rm{th}~position},  \varepsilon_{j+1},\cdots, \varepsilon_n) & \mbox{if} \ j>1, \\
    (\ast,\varepsilon_2,\cdots,  \varepsilon_n) & \mbox{if} \ j=1,
\end{array}\right.
$$
$$\Phi_j^+(\boldsymbol\varepsilon) = \left\{\begin{array}{cc}
    (\varepsilon_1,\cdots, \varepsilon_{j-1}, \underbrace{\ast}_{j\rm{th}~position},  \varepsilon_{j+1}+\varepsilon_j,\cdots, \varepsilon_n) & \mbox{if} \ j<n,\\
    (\varepsilon_1,\varepsilon_2,\cdots,  \varepsilon_{n-1},\ast) & \mbox{if} \ j=n.
\end{array}\right.
$$
We also use the same notations $\Phi_j^+$ and $\Phi_j^-$ to define an action on the function $G$ as follows:
  $$\Phi_j^-(G)(u_1,\cdots, u_n) = \left\{\begin{array}{cc}
    G(u_1,\cdots, u_{j-1}, \underbrace{u_{j-1}}_{j\rm{th}~position}, u_{j+1},\cdots, u_n) & \mbox{if} \ j>1, \\
    G(0, u_2,\cdots, u_n) & \mbox{if} \ j=1,
\end{array}\right.$$
$$\Phi_j^+(G)(u_1,\cdots, u_n) = \left\{\begin{array}{cc}
    G(u_1,\cdots, u_{j-1}, \underbrace{u_{j+1}}_{j\rm{th}~position}, u_{j+1},\cdots, u_n) & \mbox{if} \ j<n, \\
    G(u_1, u_2,\cdots, u_{n-1}, T) & \mbox{if} \ j=n.
\end{array}\right.$$

\medskip

\begin{lemma}\label{lem:by:parts}
    Let $\boldsymbol\varepsilon\in (\Z \cup \{\ast\})^n$, $j \in \{1,\dots, n\}$, and $\varepsilon_j\not\in \{0, \ast\}$.
    Let $G:\R^n\to \R$ be a differentiable function. Suppose that $\varepsilon_{j-1},\varepsilon_{j+1}\ne \ast$. Then 
\begin{equation}\label{eq_by_part1}
{\mathcal I}[\boldsymbol\varepsilon, G] = \frac{1}{2\pi i \lambda\varepsilon_j}\left({\mathcal I} [\Phi_j^-(\boldsymbol\varepsilon), \Phi_j^-(G)]-{\mathcal I} [\Phi_j^+(\boldsymbol\varepsilon),\Phi_j^+(G)]+ {\mathcal I}[\boldsymbol\varepsilon, {\partial_{u_j} G}]\right).
\end{equation}
The same is also true when $j = 1$ with $\varepsilon_2\ne\ast$ or $j = n$ with $\varepsilon_{n-1}\ne \ast$.
\end{lemma}

\begin{proof} From our assumption, we can perform integration by parts on the variable $u_j$ from $u_{j-1}$ to $u_{j+1}$.
$$
\begin{aligned}
\int_{u_{j-1}}^{u_{j+1}} e^{-2\pi i \lambda\varepsilon_j u_j} G({\bf u})~ du_j & = \frac{1}{2\pi i \lambda\varepsilon_j} \left(e^{-2\pi i \lambda\varepsilon_j u_{j-1}} \Phi_j^-G({\bf u})-e^{-2\pi i \lambda\varepsilon_j u_{j+1}} \Phi_j^+G({\bf u})\right) \\
& \quad +\frac{1}{2\pi i \lambda\varepsilon_j}\int_{u_{j-1}}^{u_{j+1}}e^{-2\pi i \lambda\varepsilon_j u_j} \partial_{u_j} G({\bf u}) ~du_j.
\end{aligned}
$$
Multiplying all other exponentials and integrating all other variables, we obtain the formula (\ref{eq_by_part1}).
\end{proof}

In the following, we will consider $n = 2q$ and $G$ is a differentiable function with $2q$ variables $(u_1,\dots, u_{2q})$. With the formula in Lemma \ref{lem:by:parts}, we can iterate the integration by parts process for $q$ different variables $u_1, u_3, \dots u_{2q-1}$. Let 
\begin{align}\label{E:Sigma^r}
\Omega_j = \{\Phi_{2j-1}^-, \Phi_{2j-1}^+ , \partial_{u_{2j-1}}\}, \quad \Omega^r = \Omega_1\times\cdots\times \Omega_r
\end{align}
where $1\le r\le q$. For each $\boldsymbol{\sigma} = (\sigma_1,\dots,\sigma_r)\in \Omega^r$, let us decompose the indices into three disjoint subsets
$$
\{1,\dots,r\} = J_1(\boldsymbol{\sigma} )\cup J_2(\boldsymbol{\sigma} )\cup J_3(\boldsymbol{\sigma} )
$$
where 
\begin{equation}\begin{gathered}\label{eqJ_123}
J_1(\boldsymbol{\sigma} ) = \{k: \sigma_k = \Phi_{2k-1}^- \}, \quad 
J_2(\boldsymbol{\sigma} ) = \{k: \sigma_k = \Phi_{2k-1}^+ \},  \\
J_3(\boldsymbol{\sigma} ) = \{k: \sigma_k = \partial_{u_{2k-1}} \}.
\end{gathered}\end{equation}
We now order 
$$
J_1(\boldsymbol{\sigma} )\cup J_2 (\boldsymbol{\sigma} )= \{k_1<k_2<\cdots<k_m\}, \ J_3 (\boldsymbol{\sigma} )= \{k_1'<\cdots< k_{\ell}'\},
$$
and define
\begin{equation}\label{eq_I_sigma}
{\mathcal I}[\boldsymbol{\sigma},G] =
{\mathcal I}_T^\lambda[\boldsymbol{\sigma},G] := (-1)^{\#J_2(\boldsymbol{\sigma} )} \cdot {\mathcal I}_T^\lambda[ \sigma_{k_m}\cdots\sigma_{k_1}(\boldsymbol{\varepsilon}), \sigma_r\cdots\sigma_2\sigma_1(G)]
    \end{equation}
(see \eqref{eq_I[e,G]} for the definition of $\mathcal{I}_T^\lambda[\boldsymbol{\varepsilon}, G]$).
For our later convenience, let us record the properties of $\mathcal{I}[\boldsymbol{\sigma},G]$ in the following lemma: 

\begin{lemma}\label{lemma_properties}
    Suppose that $\boldsymbol{\sigma}\in \Omega^r$. With reference to the above notations, we have:
    \begin{enumerate}
    \item[{\rm (i)}] $m+\ell = r$; and
    \item[{\rm (ii)}] ${\mathcal I}[\boldsymbol{\sigma},G]$ is a $2r-m$ iterated integral containing the variables of all even indices and  $u_{2j'_i-1}$ for all $i = 1,\dots, \ell$.  
    \end{enumerate}
\end{lemma}

\begin{example}
    If $n = 6$ and $\boldsymbol{\sigma} = (\Phi_{1}^+, \Phi_3^-, \partial_{u_5})$, then ${\mathcal I}[\boldsymbol{\sigma}, G]$ is equal to
    $$
    -\iiiint\limits_{0<u_2<u_4<u_5<u_6\le T} e^{-2\pi i \lambda ((\varepsilon_1+\varepsilon_2+\varepsilon_3)u_2+\varepsilon_4 u_4+\varepsilon_5 u_5+\varepsilon_6 u_6)} \partial_{u_5}G(u_2,u_2,u_2,u_4,u_5,u_6) d{\bf u}
    $$
    where $d{\bf u} = du_2\,du_4\,du_5\,du_6$.
\end{example}

Using the above notations, (\ref{eq_by_part1}) with $j=1$ can be written as
\begin{equation}\label{eq_by_part_rewrite}
{\mathcal I}[\boldsymbol\varepsilon, G] =  \frac{1}{2\pi i \lambda\varepsilon_1} \sum_{\boldsymbol\sigma\in\Omega_1} {\mathcal I}[\boldsymbol\sigma ,G].
\end{equation}
\begin{proposition}\label{prop_by_part-iterate} Let $G$ be a differentiable function on $\R^{2q}$ and $\boldsymbol{\varepsilon}\in \{-1,1\}^{2q}$. Then 
$$
{\mathcal I}[\boldsymbol\varepsilon, G] =  \left(\frac{1}{2\pi i \lambda} \right)^q\cdot \prod_{j=1}^{q}\frac{1}{\varepsilon_{2j-1}}\cdot \sum_{\boldsymbol\sigma\in\Omega^q} {\mathcal I}[\boldsymbol\sigma,G].
$$
\end{proposition}

\begin{proof}
    We prove the formula by induction on $q$. Since $\varepsilon_1\in \{-1,1\}$, $\varepsilon_1$ can be put in the denominator.  The $q = 1$ case is proved in (\ref{eq_by_part_rewrite}). Suppose the desired formula holds for some $q\ge 1$. Then, for $\boldsymbol{\varepsilon}\in\{-1,1\}^{2(q+1)}$, by induction hypothesis, we have 
\begin{equation}\label{eq_I_proof}
{\mathcal I}[\boldsymbol\varepsilon, G] =  \left(\frac{1}{2\pi i \lambda} \right)^q\cdot \prod_{j=1}^{q}\frac{1}{\varepsilon_{2j-1}} \cdot \sum_{\boldsymbol\sigma\in\Omega^q} {\mathcal I}[\boldsymbol\sigma,G].
    \end{equation}
 We now integrate by parts for the variable $u_{{2q+1}}$. Note that the operations $\Phi_{2j-1}^{\pm}$ only possibly affect the coordinates at the $2j-2$, $2j-1$, and $2j$-th positions. Hence, $\varepsilon_{2q+1}\in \{-1,1\}$ and is still the same in $\mathcal{I}[\boldsymbol\sigma,G]$  for all $\boldsymbol\sigma\in \Omega^{q}$. We can hence perform integration by parts for the variable $u_{2q+1}$.   Using (\ref{eq_I_sigma}) and applying Lemma \ref{lem:by:parts},
    $$
    \begin{aligned}
          &{\mathcal I}[\boldsymbol\sigma,G]\\
          &= (-1)^{\#J_2(\boldsymbol{\sigma} )} \cdot {\mathcal I}[ \sigma_{k_m}\cdots\sigma_{k_1}(\boldsymbol{\varepsilon}), \sigma_q\cdots\sigma_2\sigma_1(G)]\\
          &= \frac{1}{2\pi i \lambda\varepsilon_{2q+1}}\cdot (-1)^{\#J_2(\boldsymbol{\sigma} )} \left({\mathcal I} [\Phi_{2q+1}^-\sigma_{k_m}\cdots\sigma_{k_1}(\boldsymbol{\varepsilon}),\Phi_{2q+1}^- \sigma_{q}\cdots\sigma_{1}(G)]\right.\\
          & \quad \left.-{\mathcal I} [\Phi_{2q+1}^+\sigma_{k_m}\cdots\sigma_{k_1}(\boldsymbol{\varepsilon}),\Phi_{2q+1}^+ \sigma_{q}\cdots\sigma_{1}(G)]+ {\mathcal I}[\sigma_{k_m}\cdots\sigma_{k_1}(\boldsymbol{\varepsilon}), \partial_{u_{2q+1}} \sigma_q\cdots \sigma_1(G)] \right)\\
          & = \frac{1}{2\pi i \lambda\varepsilon_{2q+1}} \left({\mathcal I} [\boldsymbol{\sigma}\Phi_{2q+1}^-, G ]+{\mathcal I} [\boldsymbol{\sigma}\Phi_{2q+1}^+, G ]+{\mathcal I}[\boldsymbol{\sigma}\partial_{u_{2q+1}}, G ]\right).
    \end{aligned}
    $$
    Putting this back into (\ref{eq_I_proof}), we obtain 
    $$
    {\mathcal I}[\boldsymbol\varepsilon, G] =  \left(\frac{1}{2\pi i \lambda} \right)^{q+1} \cdot\prod_{j=1}^{q+1}\frac{1}{\varepsilon_{2j-1}} \cdot\sum_{\boldsymbol\sigma\in\Omega^{q+1}} {\mathcal I}[\boldsymbol\sigma,G]
    $$
    which is exactly our desired formula. 
\end{proof}

\section{The graph of symmetric $\alpha$-stable process}\label{sec: alpha-stable}

%Recall that $X=\{X(t) : t \ge 0\}$ is a {\bf L\'{e}vy process} if $X$ has independent and stationary increments. Let $0<\alpha\le 2$.  $X$ is called a {\bf symmmetric $\alpha$-stable process} if it is a L\'{e}vy process whose characteristic function is
%\begin{equation}\label{eq_stable-process}
%\E \left[ e^{i \xi (X(t)-X(s))}\right] = e^{- (t-s)|\xi|^{\alpha} }
%\end{equation}
%for all $t\ge s>0$ and $\xi\in \R$. When $\alpha=2$, it reduces to the standard Brownian motion. 
In this section, we will show that the graph of a symmetric $\alpha$-stable process with $\alpha\in[1,2]$ has Fourier dimension 1.  Note that if $\alpha<2$, the sample paths are not continuous and the Hausdorff dimension of its graph is equal to $\min\{1,2-\alpha^{-1}\}$ (see, e.g., \cite[Theorem 16.8]{Falconer-book}). We have the following vertical bound for the $\alpha$-stable process.

\begin{proposition}\label{prop-vertical-stable}
    Let $X$ be a symmetric $\alpha$-stable process and $\nu$ be the push-forward of the Lebesgue measure by $X$, i.e., $\nu (E) = m\left\{t\in[0,1]: X(t)\in E\right\}$.
    Then 
    $$
    \E\left[|\widehat{\nu}(\xi)|^{2q}\right] \lesssim |\xi|^{-\alpha q},
    $$
    where the implicit constant depends only on $q.$
\end{proposition}

\begin{proof}
This result should be known. For completeness, we present a quick proof that is essentially the same as the proof for the Brownian motion. For example, let us follow the proof and notations in \cite[p.160, Chapter 12]{Mattila} and take $s =1$. We only need to replace $x$ by $\xi$ and $|x|^2$ by $|\xi|^{\alpha}$ in that proof. Then we obtain our desired conclusion.
%, which appeared in \cite[p.161, Chapter 12]{Mattila}. 
\end{proof}

We now proceed to derive the horizontal bound. The symmetric $\alpha$-stable process is a self-similar process in the sense that for every $c \ne 0$,
$$
\{c X(t) : t \ge 0\} \text{ and } \{X(|c|^{\alpha} t): t \ge 0\}
$$
have the same finite dimensional distributions. Consequently, when $\xi_2\ne 0$, 
\begin{align}\label{chf:scaling}
\E\left[e^{2\pi i \xi_2 \sum_{j=1}^{2q} \varepsilon_j X(u_j)}\right] = \E\left[e^{2\pi i \sum_{j=1}^{2q}\varepsilon_j X(|\xi_2|^{\alpha}u_j)}\right]. 
\end{align}
For each $\boldsymbol{\varepsilon}\in{\mathcal A}_{2q}$ (see \eqref{A_{2q}}), let us define
$$
G_{\boldsymbol{\varepsilon}}({\bf u}) = \E\left[e^{2\pi i \sum_{j=1}^{2q}\varepsilon_j X(u_j)}\right].
$$ 
%We have the following compact formula for the expectation. 
\begin{lemma}\label{lemma4.2}
For all $q \ge 1$ and $\xi_1, \xi_2 \in \R\setminus\{0\}$, we have
\begin{equation}\label{eq_final-expect-0}
\E\left[|\widehat{\mu}(\xi_1,\xi_2)|^{2q}\right] = \frac{(q!)^2}{|\xi_2|^{2q\alpha}}\sum_{\boldsymbol{\varepsilon}\in{\mathcal A}_{2q}} {\mathcal I}[\boldsymbol{\varepsilon}, G_{\boldsymbol{\varepsilon}}],
\end{equation}
where ${\mathcal I}[\boldsymbol{\varepsilon}, G_{\boldsymbol{\varepsilon}}] = {\mathcal I}^{\lambda}_T[\boldsymbol{\varepsilon}, G_{\boldsymbol{\varepsilon}}]$ is as defined in \eqref{eq_I[e,G]} with
\begin{equation}\label{eq_lambda-T}
\lambda = \frac{\xi_1}{|\xi_2|^{\alpha}}\quad \text{and} \quad T = |\xi_2|^{\alpha}.
\end{equation}
%with $\lambda$ and $T$ in (\ref{eq_lambda-T}).
\end{lemma}

\begin{proof}
By Lemma \ref{eq_expect_simplex} and \eqref{chf:scaling}, we have 
\[
\E[|\widehat{\mu}(\xi_1,\xi_2)|^{2q}] = (q!)^2\sum_{\boldsymbol\varepsilon\in{\mathcal A}_{2q}} \int_{0\le u_1<\cdots <u_{2q}\le 1} e^{-2\pi i \xi_1 \left(\sum_{j=1}^{2q}\varepsilon_j u_j\right)}~\E[e^{2\pi i \sum_{j=1}^{2q} \varepsilon_j X(|\xi_2|^{\alpha}u_j)}] ~d{\bf u}.
\]
Then, we apply the change of variable ${\bf u} \to  |\xi_2|^{-\alpha} {\bf u}$ to see that
\[
\E[|\widehat{\mu}(\xi_1,\xi_2)|^{2q}] = \frac{(q!)^2}{|\xi_2|^{2q\alpha}}\sum_{\boldsymbol\varepsilon\in{\mathcal A}_{2q}} \int_{0\le u_1<\cdots <u_{2q}\le T} e^{-2\pi i \lambda \left(\sum_{j=1}^{2q}\varepsilon_j u_j\right)}~\E[e^{2\pi i \sum_{j=1}^{2q}\varepsilon_j X(u_j)}]~ d{\bf u}.
\]
Recalling the definitions of $G_{\boldsymbol{\varepsilon}}$ and ${\mathcal I}[\boldsymbol{\varepsilon}, G_{\boldsymbol{\varepsilon}}]$, we obtain our desired formula. 
\end{proof}

\begin{lemma}\label{lemma_never-change-signa-stable}
    Let $\boldsymbol{\varepsilon}\in{\mathcal A}_{2q}$ and $G= G_{\boldsymbol{\varepsilon}}$. 
    %Let $0<i_1<\cdots<i_{\ell}\le 2q-1$ be odd indices. 
    Then the partial derivative $G_{u_{i_1}\cdots u_{i_{\ell}}}$ does not change sign for each fixed set of odd indices $1\le i_1<\cdots<i_{\ell}\le 2q-1$.
\end{lemma}

\begin{proof}
 We can rewrite the linear combination of $X(u_j)$ as a linear combination of its increments:
\begin{align}\begin{split}\label{E:increment}
    \sum_{k=1}^{2q} \varepsilon_jX(u_j) &= (\varepsilon_1+\cdots+\varepsilon_{2q}) X(u_1)+(\varepsilon_2+\cdots+\varepsilon_{2q}) (X(u_2)-X(u_1))\\
    &\quad + \cdots +\varepsilon_{2q}(X(u_{2q})-X(u_{2q-1})).
  \end{split}  \end{align}
    By independence of increments and (\ref{eq_stable-process}) (defining $u_0 = 0$), 
    \begin{equation}\label{eq_G_e}
    G_{\boldsymbol{\varepsilon}}({\bf u}) = \prod_{j=1}^{2q} e^{-(2\pi)^{\alpha}|\varepsilon_{j}+\cdots\varepsilon_{2q}|^{\alpha}(u_j-u_{j-1})}= e^{\sum_{j=1}^{2q} a_j u_j}
    \end{equation}
    where $a_j = -(2\pi)^{\alpha} \left(|\varepsilon_j+\cdots+\varepsilon_{2q}|^{\alpha}-|\varepsilon_{j+1}+\cdots+\varepsilon_{2q}|^{\alpha}\right)$ and $a_j\ne 0$. Hence,
    $$
    (G_{\boldsymbol{\varepsilon}})_{u_{i_1}\cdots u_{i_{\ell}}}({\bf {\bf u}}) = a_{i_1}\cdots a_{i_{\ell}} G_{\boldsymbol{\varepsilon}}({\bf u}).
    $$
    Since $G_{\boldsymbol{\varepsilon}}$ is non-negative, the sign of the derivative is completely determined by the product $a_{i_1}\cdots a_{i_\ell}$. It does not change sign on the entire domain of definition. 
\end{proof}

The following theorem provides a key estimate for the horizontal bound. This can be regarded as a ``van der Corput type" lemma for oscillatory integrals (see e.g. \cite[Chapter 2]{Grafakos}) on high-dimensional simplices.  

\begin{theorem}\label{thm_never-change-sign}
    Let $G_{\boldsymbol{\varepsilon}}: \R^{2q}\to \R$ be non-negative smooth functions that are uniformly bounded for all $\boldsymbol{\varepsilon} \in \mathcal{A}_{2q}$.
    Suppose that the partial derivative $(G_{\boldsymbol{\varepsilon}})_{u_{i_1}\cdots {u_{i_\ell}}}$  does not change sign for each fixed $\boldsymbol{\varepsilon}$ and fixed set of odd indices $1\le i_1<\cdots <i_\ell\le 2q-1$.
    Then there exists $C>0$ such that for all $\boldsymbol{\varepsilon} \in {\mathcal A}_{2q}$, $T>0$ and $\lambda \in \R \setminus\{0\}$,
    $$
\left|{\mathcal I}[\boldsymbol{\varepsilon}, G_{\boldsymbol{\varepsilon}}] \right|~\le ~\frac{C^q \cdot T^q}{|\lambda|^q}
    $$   
    where ${\mathcal I}[\boldsymbol{\varepsilon}, G_{\boldsymbol{\varepsilon}}] = {\mathcal I}_T^\lambda[\boldsymbol{\varepsilon}, G_{\boldsymbol{\varepsilon}}]$ was defined as in (\ref{eq_I[e,G]}).
\end{theorem}

\begin{proof}
%As in \eqref{eq_I[e,G]}, we simply write ${\mathcal I}[\boldsymbol{\varepsilon}, G]={\mathcal I}_T^\lambda[\boldsymbol{\varepsilon}, G]$ throughout.
Let $M = \sup_{\boldsymbol{\varepsilon}\in \mathcal{A}_{2q}, {\bf u} \in \R^{2q}} |G_{\boldsymbol{\varepsilon}}({\bf u})|$.  It suffices to show that  
 \begin{equation}\label{eqclaim_thm3.2}
  |{\mathcal I}[\boldsymbol\sigma, G_{\boldsymbol{\varepsilon}}]| \le 2^q\cdot M \cdot T^q
\end{equation}
  for all $\boldsymbol{\sigma}\in\Omega^q$ and for all $ q\ge1$.    Assuming \eqref{eqclaim_thm3.2},  Proposition \ref{prop_by_part-iterate} implies that 
    $$
    \begin{aligned}
    \left|{\mathcal I}[\boldsymbol\varepsilon, G]\right| &=  \left|\left(\frac{1}{2\pi i \lambda} \right)^q\cdot \prod_{j=1}^{q}\frac{1}{\varepsilon_{2j-1}}\cdot \sum_{\boldsymbol\sigma\in\Omega^q} {\mathcal I}[\boldsymbol\sigma,G]\right|\\
    &\le \frac{6^q\cdot M \cdot  T^q}{(2\pi)^q\cdot |\lambda|^q}
    \end{aligned}
    $$ 
    since $\#\Omega^q = 3^q$. This will prove our theorem with $C = 3M \pi^{-1}$.  To see that (\ref{eqclaim_thm3.2}) holds, let us recall \eqref{eqJ_123} and write 
    $$
   J_1(\boldsymbol\sigma)\cup J_2(\boldsymbol\sigma) = \{k_1<k_2<\cdots< k_{m}\}\quad \text{and} \quad J_3(\boldsymbol\sigma) = \{k_1'<k_2'<\cdots <k_{\ell}'\}.
    $$ 
    We note that 
    $$
    |{\mathcal I}[\boldsymbol\sigma, G_{\boldsymbol{\varepsilon}}]|\le \int_{\Delta_{\sigma_{k_m}\cdots\sigma_{k_1}(\boldsymbol{\varepsilon}),T}} |(G_{\boldsymbol{\varepsilon}})_{u_{2k_1'-1}\cdots u_{2k_{\ell}'-1}}|d{\bf u}.
    $$ 
    From our assumption that the partial derivative never changes sign, we may assume without loss of generality that the partial derivative is non-negative and remove the absolute value.
    From Lemma \ref{lemma_properties} (ii),  $d{\bf u}$ consists of $du_{2k'_1-1}\cdots du_{2k'_{\ell}-1}$ and all the even indices. In particular, by the fundamental theorem of calculus, we have
    $$
    \int_{u_{2k_1'-2}}^{u_{2k_1'}} F_{u_{2k_1'-1}} ~du_{2k_1'-1} = F(\cdots ,u_{2k_1'},\cdots )- F(\cdots ,u_{2k_1'-2},\cdots )
    $$
    %at the $2k_1'-1$ position and 
    where $F = (G_{\boldsymbol{\varepsilon}})_{u_{2k_2'-1}\cdots u_{2k_{\ell}'-1}}$. 
    Applying the fundamental theorem of calculus to the variables $u_{2k_i'-1}$ for $i = 1,\dots, \ell$, we integrate away all derivatives and deduce the following:
    $$
     |{\mathcal I}[\boldsymbol\sigma, G]|\le \int G_{u_{2k_1'-1}\cdots u_{2k_{\ell}'-1}}d{\bf u} \le  2^{\ell} M \int_{[0,T]^q}  ~du_2du_4\cdots du_{2q} \le 2^q M T^q,
    $$
    where the $2^{\ell}$ arises because two terms are produced every time the fundamental theorem of calculus is applied. 
\end{proof}

We are now ready to prove Theorem \ref{th:dimF:stable}.
%\begin{theorem}
 %   Let $X(t)$ be the $\alpha$-stable process with $\alpha\in[1,2]$. Then the Fourier dimension of the graph of $X(t)$ is almost surely equal to 1. 
%\end{theorem}

\begin{proof}[{\bf \textit{Proof of Theorem \ref{th:dimF:stable}}}.]\label{proof of thm1.2}
    We will apply Proposition \ref{prop_sufficiency}. The vertical bound has been established in Proposition \ref{prop-vertical-stable}. It remains to establish the horizontal bound. Since $G_{\boldsymbol{\varepsilon}}$ does not change sign by Lemma \ref{lemma_never-change-signa-stable} and is uniformly bounded above by 1 as seen from (\ref{eq_G_e}), we can use (\ref{eq_final-expect-0}) in Lemma \ref{lemma4.2} and the result of Theorem \ref{thm_never-change-sign} to 
    %estimate $\E[|\widehat{\mu}(\xi_1,\xi_2)|^{2q}]$. We have 
    deduce that
    $$
    \E\left[|\widehat{\mu}(\xi_1,\xi_2)|^{2q}\right]\le \frac{(q!)^2}{|\xi_2|^{2q\alpha}}\cdot  (\#{\mathcal A}_{2q})\cdot \frac{C^q T^q}{|\lambda|^q}.
    $$
    Using the definitions of $\lambda$ and $T$ in (\ref{eq_lambda-T}) and $\#{\mathcal A}_{2q} =\binom{2q}{q}= \frac{(2q)!}{(q!)^2}$, we obtain 
    $$
     \E\left[|\widehat{\mu}(\xi_1,\xi_2)|^{2q}\right] \le \frac{C^q(2q)!}{|\xi_1|^q}.
    $$
    Hence, we may apply Proposition \ref{prop_sufficiency} to find that $\Fdim {\mathcal G}(X) \ge \min \{\alpha,1\}$. If $\alpha\ge 1$, then $\Fdim {\mathcal G}(X) \ge 1$. The opposite inequality  $\Fdim {\mathcal G}(X) \le 1$ is already known in \cite{FOS2014}, so $\Fdim {\mathcal G}(X) = 1$.
    Finally, it is known that $\hdim{\mathcal G}(X) = \max\{1,2-1/\alpha\}$ \cite{BG62}, so when $\alpha=1$, $\hdim{\mathcal G}(X) = 1 = \Fdim {\mathcal G}(X)$, and hence ${\mathcal G}(X)$ is a Salem set. This completes the proof. 
\end{proof}

\section{The graph of fractional Brownian motion}\label{section 5}

Let $B=\{B(t):t\ge0\}$ be a fractional Brownian motion with Hurst parameter $H\in(0,1)$.
Note that $B$ is a self-similar process in the sense that for each $c \ne 0$, $\{cB(t): t\ge 0\}$ and $\{B(|c|^{1/H}t): t \ge 0\}$ have the same finite dimensional distributions. We obtain the following lemma, which is similar to Lemma \ref{lemma4.2}.
\begin{lemma}\label{lemma5.1}
%With this notation and the notation (\ref{eq_I[e,G]}) in Section \ref{sec:bypart}, we can write 
For all $q \ge 1$ and $\xi_1, \xi_2 \in \R \setminus\{0\}$, we have
\begin{equation}\label{eq_final-expect}
\E\left[|\widehat{\mu}(\xi_1,\xi_2)|^{2q}\right] = \frac{(q!)^2}{|\xi_2|^{\frac{2q}{H}}}\sum_{\boldsymbol{\varepsilon}\in{\mathcal A}_{2q}} {\mathcal I}[\boldsymbol{\varepsilon}, G_{\boldsymbol{\varepsilon}}]
\end{equation}
where $\mathcal{A}_{2q}$ and ${\mathcal I}[\boldsymbol{\varepsilon}, G_{\boldsymbol{\varepsilon}}] = {\mathcal I}_T^\lambda[\boldsymbol{\varepsilon}, G_{\boldsymbol{\varepsilon}}]$ are, respectively, defined in \eqref{A_{2q}} and \eqref{eq_I[e,G]}, with
\begin{equation}\label{eq_lambda-T-B}
\lambda = \frac{\xi_1}{|\xi_2|^{1/H}}, \quad T = |\xi_2|^{1/H},
\end{equation}
and $G_{\boldsymbol{\varepsilon}}({\bf u}) := e^{- 2\pi^2 \Var\left( \sum_{i=1}^{2q} \varepsilon_i B(u_i)\right) }$.
%with $\lambda$ and $T$ in (\ref{eq_lambda-T})
\end{lemma}

\begin{proof}
    By Lemma \ref{eq_expect_simplex}, we have 
$$
\E\left[|\widehat{\mu}(\xi_1,\xi_2)|^{2q}\right]=(q!)^2\sum_{\boldsymbol{\varepsilon}\in{\mathcal A}_{2q}} \int_{0\le u_1<\cdots <u_{2q}\le 1} e^{-2\pi i \xi_1 \left(\sum_{j=1}^{2q}\varepsilon_j u_j\right)}\cdot G_{\boldsymbol{\varepsilon}}(\xi_2, {\bf u})~d{\bf u}
$$
where 
% \textcolor{magenta}{(the exponent $|\xi_2|^{1/H}$ is actually correct because $a^2\mathrm{Var}(X) = \mathrm{Var}(aX)$ and $\{cB(t)\}=\{B(|c|^{1/H}t\}$ in law, so there is no need to change the exponent)}
$$
\begin{aligned}
G_{\boldsymbol{\varepsilon}}(\xi_2,{\bf u}) &=\E\left[e^{2\pi i \xi_2 (\sum_{j=1}^{2q}{\varepsilon_j B(u_j) })}\right] = e^{-2\pi^2 \xi_2^2\Var\left(\sum_{j=1}^{2q}{\varepsilon_j B(u_j)}\right)}\\
&=e^{-2\pi^2\Var\left(\sum_{j=1}^{2q}{\varepsilon_j B(|\xi_2|^{1/H}u_j)}\right)}
\end{aligned}
$$
since $B$ is a self-similar Gaussian process.  By the change of variable ${\bf u}\to |\xi_2|^{-1/H}{\bf u}$,
$$
\E\left[|\widehat{\mu}(\xi_1,\xi_2)|^{2q}\right]=\frac{(q!)^2}{|\xi_2|^{\frac{2q}{H}}}\sum_{\boldsymbol{\varepsilon}\in{\mathcal A}_{2q}} \int_{0\le u_1<\cdots <u_{2q}\le T} e^{-2\pi i \lambda \left(\sum_{j=1}^{2q}\varepsilon_j u_j\right)}\cdot G_{\boldsymbol{\varepsilon}}({\bf u})~d{\bf u},
$$
where $\lambda$ and $T$ are as defined in (\ref{eq_lambda-T-B}). The last integral is exactly ${\mathcal I}_T^\lambda[\boldsymbol{\varepsilon}, G_{\boldsymbol{\varepsilon}}]$. 
This completes the proof of the lemma.
\end{proof}

Unlike the case of symmetric $\alpha$-stable processes, the derivatives of the function $G_{\boldsymbol{\varepsilon}}({\bf u})$ may change signs for fractional Brownian motion. 
Also, $G_{\boldsymbol{\varepsilon}}({\bf u})$ is only smooth in the interior of the simplex (see Lemma \ref{lem:dg} below), so Lemma \ref{lem:by:parts} and Proposition \ref{prop_by_part-iterate} cannot be applied directly.
In fact, a great deal of careful analysis is needed in order to obtain the following crucial result for proving Theorem \ref{th:dimF:fBm}.
\begin{theorem}\label{theorem_hard-estimate-introd}
    Suppose $H\in(1/2,1)$.
    Then:
    \begin{enumerate}
        \item[{\rm (1)}] There exists a constant $C \in (1,\infty)$ depending only on $H$ such that
        \begin{align}\label{E:|I|:bd}
            \left|\mathcal{I}[\boldsymbol\sigma, G_{\boldsymbol\varepsilon}]\right| \le C^{q^2} T^q,
        \end{align}
        uniformly for all $q\in\N_+$, $\boldsymbol{\varepsilon} \in \mathcal{A}_{2q}$, $\boldsymbol\sigma = (\sigma_1,\dots,\sigma_q) \in \Omega^q$, $T> 0$ and $\lambda \in \R$,
        where $\mathcal{I}[\boldsymbol\sigma, G_{\boldsymbol\varepsilon}] = \mathcal{I}_T^\lambda[\boldsymbol\sigma, G_{\boldsymbol\varepsilon}]$ is as defined in \eqref{eq_I_sigma}.\smallskip
        \item[{\rm (2)}] For any $q \in \N_+$ and $\boldsymbol{\varepsilon} \in \mathcal{A}_{2q}$, 
        \begin{align}\label{fBm:IBP}
            {\mathcal I}[\boldsymbol\varepsilon, G_{\boldsymbol{\varepsilon}}] =  \left(\frac{1}{2\pi i \lambda} \right)^q\cdot \prod_{j=1}^{q}\frac{1}{\varepsilon_{2j-1}}\cdot \sum_{\boldsymbol\sigma\in\Omega^q} {\mathcal I}[\boldsymbol\sigma,G_{\boldsymbol{\varepsilon}}].
        \end{align}
    \end{enumerate}
\end{theorem}

The proof of Theorem \ref{theorem_hard-estimate-introd} is postponed to Section \ref{S:Fourier:decay}.
Assuming Theorem \ref{theorem_hard-estimate-introd} at our disposal, we now prove our main theorem.

\begin{proof}[\bf{\textit{Proof of Theorem \ref{th:dimF:fBm} assuming Theorem \ref{theorem_hard-estimate-introd}}}] 
When $H=1/2$, Theorem \ref{th:dimF:fBm} is already known and also follows from Theorem \ref{th:dimF:stable}, so it remains to consider the case that $H \in (1/2, 1)$.
We will establish both the vertical and horizontal bounds in Proposition \ref{prop_sufficiency}. The vertical bound
$$
\E\left[|\widehat{\nu}(\xi_2)|^{2q}\right]\lesssim |\xi_2|^{-\frac{q}{H}}
$$
is well known (see \cite{Kah}). It remains to show the horizontal bound. Thanks to Lemma \ref{lemma5.1}, we have
\begin{equation}\label{eq_final-expect-fBm}
\E\left[|\widehat{\mu}(\xi_1,\xi_2)|^{2q}\right] = \frac{(q!)^2}{|\xi_2|^{\frac{2q}{H}}}\sum_{\boldsymbol{\varepsilon}\in{\mathcal A}_{2q}} {\mathcal I}[\boldsymbol{\varepsilon}, G_{\boldsymbol{\varepsilon}}],
\end{equation}
where ${\mathcal I}[\boldsymbol{\varepsilon}, G_{\boldsymbol{\varepsilon}}]$ are functions of $\lambda$ and $T$  and they depend on $\xi_1,\xi_2$ as given by \eqref{eq_lambda-T-B}.
% By Proposition \ref{prop_by_part-iterate}, 
By Part (2) of Theorem \ref{theorem_hard-estimate-introd},
$$
{\mathcal I}[\boldsymbol{\varepsilon},G_{\boldsymbol{\varepsilon}}] = \left(\frac{1}{2\pi i \lambda}\right)^q \cdot\prod_{j=1}^{q}\frac{1}{\varepsilon_{2j-1}}\cdot\sum_{\boldsymbol{\sigma}\in\Omega^q} {\mathcal I}[\boldsymbol{\sigma},G_{\boldsymbol{\varepsilon}}].
$$
Then, we may use Part (1) of Theorem \ref{theorem_hard-estimate-introd} and the definitions of $\lambda$ and $T$ in \eqref{eq_lambda-T-B} to deduce that
$$
\left|{\mathcal I}[\boldsymbol{\varepsilon},G_{\boldsymbol{\varepsilon}}]\right|\le \frac{3^qC^{q^2}T^q}{(2\pi)^q|\lambda|^q} = C_q\frac{|\xi_2|^{2q/H}}{|\xi_1|^q},
$$
where $C_q = \frac{3^qC^{q^2}}{(2\pi)^q}$. 
This, together with \eqref{eq_final-expect-fBm} and $\#{\mathcal A}_{2q} = \binom{2q}{q}$, implies that
$$
\E\left[|\widehat{\mu}(\xi_1,\xi_2)|^{2q}\right] \le \frac{(2q)!\cdot C_q}{|\xi_1|^q}.
$$
Therefore, we may invoke Proposition \ref{prop_sufficiency} to find that $\Fdim{\mathcal G}(B)\ge \min\{1,H^{-1}\} = 1$ since $H<1$. As we know the Fourier dimension of the graph of any continuous function is at most 1 (see \cite{FOS2014}).
This completes the proof of Theorem \ref{th:dimF:fBm}.
\end{proof}

The rest of the paper will be devoted to proving Theorem \ref{theorem_hard-estimate-introd}. Since $\mathcal{I}[\boldsymbol\sigma, G_{\boldsymbol\varepsilon}]$ is an integral involving derivatives of $G_{\boldsymbol\varepsilon}$ (see \eqref{eq_I_sigma}), this means that we need to consider the derivatives of $\Var(\sum_{i=1}^{2q} \varepsilon_i B(u_i))$. 
In order to prove Theorem \ref{theorem_hard-estimate-introd}, we will establish, in Section \ref{Section 6}, some properties and a combinatorial formula for $\boldsymbol{\sigma}G_{\boldsymbol{\varepsilon}}$ in terms of the derivatives.
Section \ref{sec:variance} will aim to establish the estimates of the derivatives. 
%In the next section, we need another combinatorial preparation. 

\section{Combinatorial preparations for fractional Brownian motion}\label{Section 6}

Let $\boldsymbol\sigma = (\sigma_1,\dots,\sigma_q)\in \Omega^q$. 
Recall the notations in \eqref{eqJ_123} with $r=q$.
The indices $k_1<\cdots<k_m$ and $k_1'<\cdots<k_{\ell}'$ are defined such that $m+\ell=q$,
\begin{equation}\label{eq_J_12-section7}
    J_1(\boldsymbol\sigma)\cup J_2(\boldsymbol\sigma)  = \{k: \sigma_k = \Phi_{2k-1}^+ \text{ or } \Phi_{2k-1}^-\}=\{k_1,\dots,k_{m}\},
\end{equation}
    and
\begin{equation}\label{eq_J_3-section7}
    J_3(\boldsymbol\sigma) = \{k:\sigma_k=\partial_{u_{2k-1}}\}=\{k'_1,\dots,k'_{\ell}\}.
\end{equation}
Furthermore, $\{k_1,\dots, k_m,k_1',\dots,k_{\ell}'\} = \{1,\dots, q\}$. 
    For simplicity, we will use the following notation:
    \[
        \boldsymbol{\sigma}G = \sigma_q \cdots \sigma_1 G,
    \]
    where the notations $\Phi_{j}^{\pm}G$ are defined in Section \ref{sec:bypart} when $\sigma_j = \Phi_j^{\pm}$.
    Also, recall from \eqref{eq_I[e,G]} and \eqref{eq_I_sigma} that
    \begin{align}\begin{split}\label{E:I[e,G]}
        \mathcal{I}[\boldsymbol{\sigma}, G_{\boldsymbol{\varepsilon}}]
        &= (-1)^{\# J_2(\boldsymbol{\sigma})} \cdot \mathcal{I}[\sigma_{k_m} \cdots \sigma_{k_1}\boldsymbol{\varepsilon}\,,
        \boldsymbol{\sigma} G_{\boldsymbol{\varepsilon}}]\\
        &= (-1)^{\# J_2(\boldsymbol{\sigma})} \int_{\Delta_{\sigma_{k_m} \cdots \sigma_{k_1}\boldsymbol{\varepsilon},T}} 
        e^{-2\pi i \lambda \langle \sigma_{k_m} \cdots \sigma_{k_1}\boldsymbol{\varepsilon}, {\bf u} \rangle}
        \boldsymbol{\sigma}   G_{\boldsymbol{\varepsilon}}({\bf u}) \, d{\bf u}.
    \end{split}\end{align}
 %We need to perform multi-variate and combinatorial differentiations of $G_{\boldsymbol{\varepsilon}}$ and
 For $\boldsymbol{\varepsilon} \in (\Z\cup\{\ast\})^{2q}$ and $u_1,\dots,u_{2q}\ge0$, define
\begin{align}\label{E:G_eps}
    G_{\boldsymbol{\varepsilon}}({\bf u}) := \exp\left\{ - 2\pi^2 \Var\left( \sum_{1 \le i \le 2q : \varepsilon_i \ne \ast} \varepsilon_i B(u_i)\right) \right\}.
\end{align}
 %From this section on, we will use $\exp(x)$ to denote the exponential function for a clearer display of equations.
 The goal of this section is to study the properties of $\boldsymbol{\sigma}G_{\boldsymbol{\varepsilon}}$ and establish, using Fa\`a di Bruno's formula, a combinatorial formula for $\boldsymbol{\sigma}G_{\boldsymbol{\varepsilon}}$ in terms of the derivatives of the variance function in \eqref{E:G_eps}.
 
\subsection{Non-trivial entries and variables}
First, we study some basic properties for the action of $\boldsymbol{\sigma}$ on the function $G_{\boldsymbol{\varepsilon}}$ and the dependency of $\boldsymbol{\sigma}G_{\boldsymbol{\varepsilon}}$ on its variables.
%We see that (\ref{E:G_eps}) is a composition of two functions and we will need to estimate its mixed partial derivatives. 
%We first need to study the combinatorial properties of $\boldsymbol\sigma$ under differentiation. 
%In particular, we aims to extract our the necessary variables inside $\boldsymbol{\sigma}G_{\boldsymbol{\varepsilon}}$ required for the differentiation and integration. 

\begin{lemma}\label{lem:sig:sig:G}
Let $q \in \N_+$, $\boldsymbol\sigma = (\sigma_1, \dots, \sigma_q) \in \Omega^q$, and ${\bf u}\in \R^{2q}$.
Suppose $G: \R^{2q} \to \R$ is smooth in a neighborhood of ${\bf u}$.
Then for $i \ne j$,
    \begin{align}\label{E:sig:sig:G}
        \sigma_i \sigma_j G({\bf u}) = \sigma_j \sigma_i G({\bf u}).
    \end{align}
\end{lemma}

\begin{proof}
    %Recall that $\Omega^r = \Omega_1 \times \cdots \times \Omega_r$, where $\Omega_i = \{\Phi_{2i-1}^-, \Phi_{2i-1}^+ , \partial_{u_{2i-1}}\}$.
    Recall \eqref{E:Sigma^r}.
    Without loss of generality, we may and will assume that $i<j$.
    For the case that $\sigma_i = \partial_{u_{2i-1}}$ and $\sigma_j = \partial_{u_{2j-1}}$, \eqref{E:sig:sig:G} is simply symmetry of mixed derivatives, which follows from the smoothness of $G$ near ${\bf u}$.
    Next, consider the case that at least one of $\sigma_i, \sigma_j$ is not a differential operator.  Indeed, if $j \ne i+1$, it is clear that \eqref{E:sig:sig:G} holds because $\Phi_{2i-1}^\pm G({\bf u})$ transforms the variable $u_{2i-1}$ to either $u_{2i}$ or $u_{2i-2}$ and leaves all other variables unchanged.
    If $j=i+1$, we can directly check all nine combinations from $\Omega_i\times\Omega_{i+1}$ by definition. For example,     \begin{align*}
        &\Phi_{2i-1}^+ \partial_{u_{2i+1}} G({\bf u})
        = \partial_{u_{2i+1}}G(\dots, u_{2i-2}, u_{2i}, u_{2i}, u_{2i+1}, u_{2i+2}, \dots) 
        = \partial_{u_{2i+1}} \Phi_{2i-1}^+ G({\bf u}),\\
         &\Phi_{2i-1}^+ \Phi_{2i+1}^+ G({\bf u})
        = G(\dots, u_{2i-2}, u_{2i}, u_{2i}, u_{2i+2}, u_{2i+2}, \dots) 
        = \Phi_{2j-1}^+ \Phi_{2i-1}^+ G({\bf u}),
       % &\Phi_{2i-1}^- \partial_{u_{2i+1}} G({\bf u})
        %= \partial_{u_{2i+1}} G(\dots, u_{2i-2}, u_{2i-2}, u_{2i}, %u_{2i+1}, u_{2i+2}, \dots) 
        %= \partial_{u_{2i+1}} \Phi_{2i-1}^- G({\bf u}),\\
        %&\partial_{u_{2i-1}} \Phi_{2i+1}^+ G({\bf u})
        %= \partial_{u_{2i-1}} G(\dots, u_{2i-2}, u_{2i-1}, u_{2i}, %u_{2i+2}, u_{2i+2}, \dots) 
        %= \Phi_{2i+1}^+ \partial_{u_{2i-1}} G({\bf u}),\\
        %&\partial_{u_{2i-1}} \Phi_{2i+1}^- G({\bf u})
        %= \partial_{u_{2i-1}} G(\dots, u_{2i-2}, u_{2i-1}, u_{2i}, u_{2i}, u_{2i+2}, \dots) 
        %= \Phi_{2i+1}^- \partial_{u_{2i-1}} G({\bf u}),
    \end{align*}
    %and
    %\begin{align*}
      %  &\Phi_{2i-1}^+ \Phi_{2i+1}^+ G({\bf u})
      %  = G(\dots, u_{2i-2}, u_{2i}, u_{2i}, u_{2i+2}, u_{2i+2}, \dots) 
      %  = \Phi_{2j-1}^+ \Phi_{2i-1}^+ G({\bf u}),\\
      %  &\Phi_{2i-1}^+ \Phi_{2i+1}^- G({\bf u})
      %  = G(\dots, u_{2i-2}, u_{2i}, u_{2i}, u_{2i}, u_{2i+2}, \dots) 
      %  = \Phi_{2i+1}^- \Phi_{2i-1}^+ G({\bf u}),\\
       % &\Phi_{2i-1}^- \Phi_{2i+1}^+ G({\bf u})
      %  = G(\dots, u_{2i-2}, u_{2i-2}, u_{2i}, u_{2i+2}, u_{2i+2}, %\dots) 
      %  = \Phi_{2i+1}^+ \Phi_{2i-1}^- G({\bf u}),\\
      %  &\Phi_{2i-1}^- \Phi_{2i+1}^- G({\bf u})
       % = G(\dots, u_{2i-2}, u_{2i-2}, u_{2i}, u_{2i}, u_{2i+2}, \dots) 
       % = \Phi_{2i+1}^- \Phi_{2i-1}^- G({\bf u}),
  %  \end{align*}
   where we set $u_{-1}=0$ and $u_{2q+2}=T$ for convenience. This proves the lemma.
\end{proof}

\begin{definition}
Let $\boldsymbol{\varepsilon} = (\varepsilon_j)_{j=1}^{2q} \in (\Z\cup\{\ast\})^{2q}$. We say that an entry $\varepsilon_j$ of $\boldsymbol{\varepsilon}$ is {\bf non-trivial} if $\varepsilon_j\ne0$ and $\varepsilon_j\ne\ast$.    
\end{definition}
 The next lemma counts the number of non-trivial entries when  applying  $\boldsymbol\sigma\in\Omega^q$ that consists of $\Phi_j^\pm$-type operations to $\boldsymbol\varepsilon\in{\mathcal A}_{2q}$. 
 %Recall from \eqref{A_{2q}} that ${\mathcal A}_{2q}$ is the collection of $(\varepsilon_j)_{j=1}^{2q}\in\{-1,1\}^{2q}$ such that $\sum_{j=1}^{2q} \varepsilon_j= 0$. 
 Recall \eqref{A_{2q}} for the definition of $\mathcal{A}_{2q}$.

\begin{lemma}\label{lem:var:reduc}
    Let $\boldsymbol\varepsilon \in{\mathcal A}_{2q}$ and $\boldsymbol{\sigma}\in \Omega^q$ with \eqref{eq_J_12-section7} and \eqref{eq_J_3-section7}. For $1 \le i \le m$, let $\sigma_{k_i}=\Phi_{2k_i-1}^{\diamond_i}$, where $\diamond_i \in \{+, -\}$.
    Suppose that $\boldsymbol{\alpha}:=\sigma_{k_m}\cdots\sigma_{k_1}\boldsymbol\varepsilon$ has exactly $I$ non-trivial entries  and the indices of these non-trivial entries are $j_1<\dots< j_I$.
    Then the following properties hold.
    \begin{enumerate}
        \item[{\rm (i)}] The following set inclusions hold for the indices:
        \begin{align*}
            &\{1,\cdots,2q\} \setminus\{2k_1-1, 2k_1-1\diamond_1 1, \cdots, 2k_m-1, 2k_m-1\diamond_m 1\}\\
            &\subset \{j_1, \cdots, j_I\}\\
            &\subset \{1,\cdots,2q\} \setminus \{2k_1-1, \dots, 2k_m-1\}
        \end{align*}
        and hence $2q-2m\le I\le 2q-m$;
        \item[{\rm (ii)}] $\sigma_{k_m} \cdots \sigma_{k_1} G_{\boldsymbol{\varepsilon}}({\bf u}) = G_{\sigma_{k_m}\cdots \sigma_{k_1} \boldsymbol{\varepsilon}}({\bf u})$ and is a function of the variables $u_{j_1},\dots, u_{j_I}$, independent of the other variables $u_j$ where $j \not\in\{j_1,\dots,j_I\}$;
        \item[{\rm (iii)}] If we write  
        \begin{equation}\label{eq_G_m}
        G_{\sigma_{k_m}\cdots \sigma_{k_1} \boldsymbol{\varepsilon}}({\bf u}) = \exp\left\{-2\pi^2 \Var\left(\sum_{i=1}^Ia_i B(u_{j_i})\right)\right\},
        \end{equation}
        then 
        $$
        \sum_{i=1}^I a_i =\sum_{i=1}^{2q} \varepsilon_i= 0, 
        $$
        and $|a_i|\le 3$ for all $i = 1,\dots, I$. 
    \end{enumerate}
\end{lemma}

To be clear about our notation, here $\diamond_i\in\{+,-\}$ denotes an arithmetic operation so that $$2k_i-1\diamond_i1 =  \begin{cases}
            2k_i & \text{if } \diamond_i=+,\\
            2k_i-2 & \text{if }  \diamond_i=-.
        \end{cases}$$

\begin{proof}
(i) Note that, from the definition of $\Phi_{2k_i-1}^\pm$, when $\Phi_{2k_i-1}^{\diamond}$ acts on a vector, the entry at the $(2k_i-1)$-th position becomes $\ast$, so $2k_i-1$ is an index of a trivial entry. 
The vector $\boldsymbol{\alpha}$ is obtained by applying $\Phi_{2k_1-1}^{\diamond_1}, \dots, \Phi_{2k_m-1}^{\diamond_m}$ to $\boldsymbol{\varepsilon}$, so the $\ast$ entries of $\boldsymbol{\alpha}$ appear exactly at the $2k_1-1, \dots, 2k_m-1^{\rm th}$ positions, and all other entries are numbers.
Hence,  $\{j_1,\cdots, j_I\}\subset \{1,\cdots,2q\} \setminus \{2k_1-1, \dots, 2k_m-1\}$. 

When $\Phi_{2k_i-1}^{\diamond_i}$ is applied to $\boldsymbol{\varepsilon}$, the $(2k_i-1\diamond_i1)$-th entry becomes $\varepsilon_{2k_i-1\diamond_i1}+\varepsilon_{2k_i-1}$, which may or may not be 0.
%This may becomes zero, which means it is trivialized. 
For all $j\not\in \{2k_i-1,2k_i-1\diamond_i1:i=1,\dots ,m\}$, the $j$-th entry of $\boldsymbol{\varepsilon}$ is unaffected by the actions of the $\Phi_{2k_i-1}^{\diamond_i}$'s,
%then $j$ cannot be changed by $\Phi_{2k_i-1}^{\ast}$, 
meaning that the entry is still $\varepsilon_j$, which is equal to $\pm 1$. This shows the first inclusion.

   Next, we prove (ii) and (iii) by induction on $m$. Suppose $m=1$ and $\sigma_{k_1} = \Phi_{2k_1-1}^{\diamond}$, where $\diamond \in \{+, -\}$, then the $(2k-1)$-th entry of $\sigma_{k_1} \boldsymbol\varepsilon$ is $\ast$, the $(2k_1-1\diamond 1)$-th entry is $\varepsilon_{2k_1-1\diamond 1} + \varepsilon_{2k-1}$, which is either $\pm 2$ or $0$, and all other entries are unchanged and equal to $\pm 1$.
    
    {\bf Case 1}: If $\varepsilon_{2k_1-1\diamond 1} + \varepsilon_{2k-1} = \pm 2$, then
    $\{j_1, \dots, j_I\} = \{1,\dots, 2q\} \setminus\{2k_1-1\}$ and
    \begin{align*}
        \sigma_{k_1} G_{\boldsymbol{\varepsilon}}({\bf u})
        &= \exp\left\{ - 2\pi^2 \Var\left( \sum_{i \ne 2k_1-1, 2k_1-1\diamond 1} \varepsilon_i B(u_i) + (\varepsilon_{2k_1-1\diamond 1} + \varepsilon_{2k-1}) B(u_{2k_1-1\diamond 1}) \right) \right\}\\
        &= G_{\sigma_{k_1}\boldsymbol{\varepsilon}}({\bf u}).
    \end{align*}
    
    {\bf Case 2}: If $\varepsilon_{2k_1-1\diamond 1} + \varepsilon_{2k-1} = 0$, then
    $\{j_1, \dots, j_I\} = \{1,\dots, 2q\} \setminus\{2k_1-1, 2k_1-1\diamond 1\}$ and
    \begin{align*}
        \sigma_{k_1} G_{\boldsymbol{\varepsilon}}({\bf u})
        = \exp\left\{ - 2\pi^2 \Var\left( \sum_{i \ne 2k_1-1, 2k_1-1\diamond 1} \varepsilon_i B(u_i) \right) \right\}
        = G_{\sigma_{k_1}\boldsymbol{\varepsilon}}({\bf u}).
    \end{align*}
    In both cases, properties (ii) and (iii) hold since $\sum_{j=1}^{2q}\varepsilon_j = 0$ by the assumption that $\boldsymbol{\varepsilon}\in {\mathcal A}_{2q}$ (see \eqref{A_{2q}}).

    Suppose the lemma is valid up to certain $m\ge 1$. Let $\boldsymbol{\alpha}:= \sigma_{k_m} \cdots \sigma_{k_1} \boldsymbol{\varepsilon}$. According to (\ref{eq_G_m}), $(a_1,\dots ,a_I)$ are the non-trivial entries of $\boldsymbol{\alpha}$, and the trivial entries are among the positions $\{2k_i-1,2k_i-1\diamond_i1:i=1,\dots, m\}$, thanks to property (i).

    Consider $k_{m+1} \in \{1, \dots, q\}$ such that $k_{m+1} > k_m$ and $\sigma_{k_{m+1}} = \Phi_{2k_{m+1}-1}^{\diamond_{m+1}}$, where $\diamond_{m+1} \in \{+, -\}$. Note that 
    \begin{equation}\label{eq-equality}
    2k_{m+1}-1\diamond_{m+1}1\ge 2k_m-1\diamond_m 1
    \end{equation}
    and equality holds if and only if  $k_{m+1} = k_m+1$, $\diamond_{m+1}= -$ and $\diamond_m = +$. 
    
    Suppose that the equality in \eqref{eq-equality} does not hold. 
    %Then $2k_{m+1}-2,2k_{m+1}-1,2k_{m+1}$ are still indices of non-trivial entries in $\boldsymbol{\alpha}$ and $\varepsilon_{2k_{m+1}-2}, \varepsilon_{2k_{m+1}-1},\varepsilon_{2k_{m+1}}$ are one of the $a_i$. 
    When applying $\sigma_{k_{m+1}}$ to the vector ${\boldsymbol{\alpha}} $, all entries remain unchanged except for the $(2k_{m+1}-1)$-th entry (which becomes $\ast$) and possibly the $(2k_{m+1}-1\diamond_{m+1}1)$-th entry (which becomes ${\varepsilon}_{2k_{m+1}-1\diamond_{m+1}1}+ \varepsilon_{2k_{m+1}-1}$ and may or may not be 0).
  
%    Let $\mathcal{J} = \{1,\cdots I\} \setminus\{2k_{m+1}-1, 2k_{m+1}-1\diamond_{m+1}1\}$.
    Then, by the induction hypothesis, if $G_{\boldsymbol{\alpha}}({\bf u}) = \exp(-2\pi^2\Var(\sum_{i=1}^{I}a_iB(u_{j_i})))$ with $\sum_{i=1}^Ia_i=0$, then 
    \begin{align}
        &\boldsymbol{\sigma}G_{\boldsymbol{\varepsilon}}({\bf u})
        = \sigma_{k_{m+1}} G_{\boldsymbol{\alpha}}({\bf u})\notag\\
        &= \exp\left\{ -2\pi^2 \Var\left( \sum_{i:j_i\ne u,v} a_i B(u_{j_i}) + ({\varepsilon}_{2k_{m+1}-1\diamond_{m+1}1}+ \varepsilon_{2k_{m+1}-1}) B(u_{2k_{m+1}-1\diamond_{m+1}1}) \right) \right\}\notag\\
        &= G_{\sigma_{k_{m+1}} \boldsymbol{\alpha}}({\bf u}),
        \label{sigma:G(u)}
    \end{align}
    where $u = 2k_{m+1}-1$ and $v=2k_{m+1}-1\diamond_{m+1}1$.  %with  $a_v = \varepsilon_{2k_{m+1}-1\diamond_{m+1}1}+\varepsilon_{2k_{m+1}-1}$. 
    This shows property (ii). To prove property (iii), we first note that the induction hypothesis implies that $\sum_{i=1}^Ia_i=0$. Also, note that $a_u =\varepsilon_{2k_{m+1}-1}$ and $a_v = \varepsilon_{2k_{m+1}-1\diamond_{m+1}1}$. In particular, we have the following relation:
    $$
    \sum_{i:j_i\ne u,v} a_i + (\varepsilon_{2k_{m+1}-1\diamond_{m+1}1}+\varepsilon_{2k_{m+1}-1}) = \sum_{i=1}^I a_i = 0. 
    $$ 
    Hence, property (iii) follows from \eqref{sigma:G(u)}, the preceding relation, and the observation that all coefficients of $B(\cdot)$ in \eqref{sigma:G(u)} are at most 3 by induction hypothesis and $|{\varepsilon}_{2k_{m+1}-1\diamond_{m+1}1}+ \varepsilon_{2k_{m+1}-1}|\le 2$.

    Suppose that the equality holds in (\ref{eq-equality}). Then $2k_{m+1}-2$ coincides with $2k_m$ and
    \begin{align*}
        &\boldsymbol{\sigma}G_{\boldsymbol{\varepsilon}}({\bf u})\\
        &=\exp\left\{ -2\pi^2 \Var\left( \sum_{i:j_i \ne u,v,w} a_i B(u_{j_i}) + (\varepsilon_{2k_m-1}+\varepsilon_{2k_m}+\varepsilon_{2k_{m+1}-2}) B(u_{2k_{m+1}-1}) \right) \right\}\\
        &= G_{\sigma_{k_{m+1}}\boldsymbol{\alpha}}({\bf u}),
    \end{align*}
    where $u = 2k_{m+1}-1$, $v = 2k_m$, and $w = 2k_{m+1}-1$.
    %the coefficient at $2k_{m+1}-2$ becomes 
    Note that $\varepsilon_{2k_m-1}+\varepsilon_{2k_m}+\varepsilon_{2k_{m+1}-1}$ must be non-zero and has absolute value at most 3 since $\varepsilon_j\in\{1,-1\}$. 
    It is now easy to check properties (ii) and (iii). 
    %In particular, $|a_i|\le 3$ and it equals $3$ if  $\varepsilon_{2k_m-1},\varepsilon_{2k_m},\varepsilon_{2k_{m+1}-2}$ are all $1$ or all $-1$. 
    This completes the proof of the lemma. 
\end{proof}

\begin{rem}\label{rmk:entry} 
Let $\boldsymbol{\sigma} = (\sigma_1,\cdots,\sigma_q)\in \Omega^q$ with $k_1<\cdots<k_m$ and $k_1'<\cdots<k_{\ell}'$ defined by (\ref{eq_J_12-section7}) and (\ref{eq_J_3-section7}), and let $\boldsymbol\alpha= \sigma_{k_m}\cdots\sigma_{k_1}\boldsymbol\varepsilon$.
Then Lemma \ref{lem:var:reduc} shows the following:
\begin{enumerate} 
    \item[(i)] If $\sigma_{k_i'}$ is a differential operator, the ($2k_i'-1$)-th the entry of $\boldsymbol\alpha$ must be  non-trivial;
    \item[(ii)] When $\sigma_{k_i} = \Phi_{2k_i-1}^{\diamond}$ ($\diamond\in\{+,-\}$) is applied to $\boldsymbol\varepsilon$, it will
    trivialize at most two entries. In fact, the $(2k_i-1)$-th entry will become $\ast$ and the $(2k_i-1\diamond1)$-th entry may become zero if $\varepsilon_{2k_i-1} + \varepsilon_{2k_i-1\diamond 1} = 0$. 
\end{enumerate}
\end{rem}

\begin{definition}\label{def_non-trivial-Variable}
    %Let $\boldsymbol\varepsilon\in{\mathcal A}_{2q}$ and $\boldsymbol{\sigma} = (\sigma_1,\cdots,\sigma_q)\in \Omega^q$  with $k_1<\cdots<k_m$ and $k_1'<\cdots<k_{\ell}'$ defined in (\ref{eq_J_12-section7}) and (\ref{eq_J_3-section7}). 
    Let $\boldsymbol{\varepsilon}\in\mathcal{A}_{2q}$ and $\boldsymbol{\sigma}\in \Omega^q$.
    With the notations in Remark \ref{rmk:entry}, define $\boldsymbol\alpha= \sigma_{k_m}\cdots\sigma_{k_1}\boldsymbol\varepsilon$. Suppose $\boldsymbol{\alpha}$ has exactly $I$ non-trivial entries with indices $j_1<\cdots < j_I$. We call the variables $s_1= u_{j_1},\dots, s_I= u_{j_I}$ {\bf non-trivial variables}. 
    We call the variables $u_{2k_1'-1}, \dots, u_{2k_\ell'-1}$ the {\bf differentiation variables}. 
    From Remark \ref{rmk:entry} (i), we know that all differentiation variables are non-trivial variables, so we can define a subset of indices $p_1<\cdots<p_\ell$ in $\{1,\dots, I\}$ such that $s_{p_i}= u_{2k_i'-1}$.
\end{definition}

%The following lemma aims to prove that the non-trivial entries for the differential operators must differ at least 2.   
The next lemma shows that the indices of the differentiation variables $s_{p_1},\dots, s_{p_\ell}$ must differ at least by 2.
%This property will be used in the later estimate. 

\begin{lemma}\label{lem:dvar:sep}
$p_{i+1} - p_i \ge 2$ for all $1 \le i < \ell$.
   % Let $\boldsymbol{\varepsilon} \in \mathcal{A}_{2q}$ and $\boldsymbol{\sigma} = (\sigma_1, \dots, \sigma_q)\in \Omega^q$ with $k_1<\cdots<k_m$ and $k_1'<\cdots<k_{\ell}'$ defined in (\ref{eq_J_12-section7}) and (\ref{eq_J_3-section7}). Let $\boldsymbol{\alpha} = \sigma_{k_m} \cdots \sigma_{k_1} \boldsymbol{\varepsilon}$ and   $j_1,\cdots, j_I\in \{1,\dots,2q\}$ be the indices of the non-trivial entries of $\boldsymbol\alpha$ in increasing order. Denote by  $p_1 < \cdots < p_m$ the indices such that $\{ j_{p_1}, \dots, j_{p_m} \} = \{ 2k_1'-1, \dots, 2k_\ell'-1 \}$.    Then  %that is, there is at least one of $\alpha_{j_{p_i}+1},\cdots,\alpha_{j_{p_{i+1}}-1}$ is a non-zero, non-$\ast$ entry of $\boldsymbol{\alpha}:=\sigma_{k_\ell} \cdots \sigma_{k_1} \boldsymbol{\varepsilon}$.
\end{lemma}

\begin{proof}
 %Fix $j_{p_i} = 2k_i'-1$ and $j_{p_{i+1}} = 2k_{i+1}'-1$. 
 If $k'_{i+1} = k'_i+1$, then $\sigma_{k_i'}$ and $\sigma_{k_i'+1}$ are both differential operators. Hence, $\varepsilon_{2k_i'-1}$, $ \varepsilon_{2k_i'}$, $\varepsilon_{2k_i'+1}$ are all non-trivial entries and thus $p_{i+1}-p_i=2$ in this case.

If $k'_{i+1} > k'_i+1$, then the sequence $\sigma_{k_i}, \sigma_{k_i+1}, \dots, \sigma_{k_{i+1}}$ is of the form
%$\sigma_j$ for $j = k_i,k_i+1,\cdots, k_{i+1}$,
\begin{equation}\label{eq_sigma-between}
\partial_{u_{2k_i'-1}}, ~ \Phi_{2k_i'+1}^{\pm},~ \dots, ~\Phi_{2k_{i+1}'-3}^{\pm},~ \partial_{u_{2k_{i+1}'-1}}.
\end{equation}
%Note that the vector $(\varepsilon_{2k_i'},\dots, \varepsilon_{2k_{i+1}'-2})$ contains $2(k_{i+1}'-k_i')-1$ entries. 
By Remark \ref{rmk:entry} (ii), each $\Phi_{2j-1}^{\pm}$ trivializes at most two entries adjacent to each other, so all of the $\Phi_{2j-1}^\pm$'s in (\ref{eq_sigma-between}) trivialize at most $2(k_{i+1}'-k_i')-2$ many entries. Hence, at least one non-trivial entry remains in between $2k_i'-1$ and $2k_{i+1}'-1$, showing that $p_{i+1}-p_i\ge 2$. 
\end{proof}

\subsection{Derivative formulae for the variance.} Let $I \in \N_+$, ${\bf a} = (a_1, \dots, a_I) \in \R^I$, and $t_1,\dots,t_I\ge0$.
%Define
%\begin{align}\label{E:G_eps}
%    G_{\varepsilon}({\bf u}) := \exp\left\{ - \pi \Var\left( \sum_{i=1}^{2q} \varepsilon_i B(u_i)\right) \right\}.
%\end{align}
By $\Var\left(\sum a_i X_i\right) = \sum a_ia_j \mathrm{Cov}(X_i,X_j)$, and the covariance formula \eqref{E:fBm:cov}, we have
\begin{align}\begin{split}\label{E:fBm:var}
    &\Var\left( \sum_{i=1}^{I} a_i B(t_i)\right) 
    = \sum_{i=1}^{I} \sum_{j=1}^{I} a_i a_j \mathrm{Cov}(B(t_i),B(t_j))\\
    &= \frac12 \sum_{i=1}^{I} \sum_{j=1}^{I} a_i a_j \left( t_i^{2H} + t_j^{2H} - |t_i-t_j|^{2H} \right).
\end{split}\end{align}
Define $g_{\bf a} : [0, \infty)^I \to \R$ by
\begin{align}\label{E:g_a(s)}
    g_{\bf a}({\bf s}) = g_{\bf a}(s_1,\dots, s_I) = -\Var\left( \sum_{i=1}^I a_i B(s_i) \right).
\end{align}
%where  ${\bf s} = (s_1,\dots, s_I)$. 
The lemma below provides formulae for the derivatives of $g_{\bf a}({\bf s})$ for our later use.
%We first record the derivatives formulae we need to use in the following:
\begin{lemma}\label{lem:dg}
  Let $0<s_1<\dots<s_I$ and $a=(a_1,\dots,a_I)\in \R^I$.
  Suppose $\sum_{i=1}^I a_i = 0$.
  Then, we have the following:
\begin{enumerate}
    \item[(i)] For any $i\in \{1,\dots,I\}$,
    \begin{align*}
        \partial_{s_i} g_a({\bf s}) 
        = 2H\sum_{j=1}^{i-1} a_i a_j (s_i-s_j)^{2H-1} - 2H \sum_{j=i+1}^{I} a_i a_j (s_j-s_i)^{2H-1}.
    \end{align*}
    \item[(ii)] For any $i<j$ in $\{1,\dots,I\}$,
    \begin{align*}
        \partial_{s_i} \partial_{s_j} g_a({\bf s}) 
        = -2H(2H-1) a_i a_j (s_{j}-s_{i})^{2H-2}.
    \end{align*}
    \item[(iii)] If $k \ge 3$, then for any $i_1 < \dots < i_k$ in $\{1, \dots, I\}$, 
    \begin{align*}
        \partial_{s_{i_1}} \cdots \partial_{s_{i_k}} g_a({\bf s}) = 0.
    \end{align*}
\end{enumerate}  
\end{lemma}

\begin{proof}
   The assumption $\sum_{i=1}^I a_i = 0$ implies that
    $$
    \sum_{i=1}^I\sum_{j=1}^I a_ia_j (s_i^{2H}+s_j^{2H})=     \sum_{i=1}^Ia_is_i^{2H}\left(\sum_{j=1}^I a_j\right)+     \sum_{j=1}^Ia_js_j^{2H}\left(\sum_{i=1}^I a_i\right)=0.
    $$
    Hence, by \eqref{E:fBm:var},
    \[
    g_a({\bf s}) = \sum_{1\le i<j\le I} a_ia_j (s_j-s_i)^{2H}.
    \]
    Then, all the derivative formulae follow by direct calculations.
\end{proof}

\subsection{Fa\`a di Bruno's formula} The general chain rule for high-order derivatives is known as {\it Fa\`a di Bruno's formula}. 
We will use a multi-variate, combinatorial version of Fa\`a di Bruno's formula, which is given below.

\begin{lemma}\cite[Proposition 1]{H06}\label{lem:fdb}
    Suppose $f: \R \to \R$ is a smooth function and $g: \R^d \to \R$ is a function that is smooth in a neighborhood of ${\bf x}\in \R^d$.
    Then,
    \begin{align*}
        \partial_{x_1} \cdots \partial_{x_d} f(g({\bf x}))
        = \sum_{P\in \mathscr{P}\{1,\dots,d\}} f^{(\# P)}(g({\bf x})) \prod_{B \in P} g^{(B)}({\bf x}),
    \end{align*}
    where $\mathscr{P}\{1,\dots,d\}$ denotes the set of all partitions of $\{1,\dots,d\}$, $\# P$ denotes the cardinality of $P$, and 
    \[
        g^{(B)}({\bf x}) = \partial_{x_{i_1}} \cdots \partial_{x_{i_n}} g({\bf x})
    \]
    whenever $\# B = n$ and $B = \{i_1,\cdots,i_n\}$.
\end{lemma}

\begin{proposition}\label{Lemma_deBruno}
      Let $\boldsymbol{\varepsilon} \in \mathcal{A}_{2q}$ and $\boldsymbol{\sigma} = (\sigma_1, \dots, \sigma_q)\in \Omega^q$ with $k_1<\cdots<k_m$ and $k_1'<\cdots<k_{\ell}'$ defined by \eqref{eq_J_12-section7} and \eqref{eq_J_3-section7}. Let $j_1<\cdots <j_I$ be the indices of the non-trivial entries of $\sigma_{k_m}\cdots\sigma_{k_1}\boldsymbol{\varepsilon}$. Then we have the following:
      \begin{enumerate}
      \item[{\rm (i)}] $\boldsymbol{\sigma}G_{\boldsymbol{\varepsilon}} ({\bf u})$ is a function depending only on the non-trivial variables $u_{j_1},\dots, u_{j_I}$;
      \item[{\rm (ii)}] Fa\`a di Bruno's formula: with the notations in Definition \ref{def_non-trivial-Variable},
      \begin{align}\label{E:fdb}
        \boldsymbol\sigma G_{\boldsymbol\varepsilon}({\bf s})
        = \sum_{P\in \mathscr{P}_2\{{p_1},\dots,{p_{\ell}}\}} (2\pi^2)^{\# P} \exp(2\pi^2 g_{\bf a}({\bf s})) \prod_{B \in P} g_{\bf a}^{(B)}({\bf s}),
    \end{align}
    where 
    \begin{align*}
        \qquad
        \mathscr{P}_2\{{p_1},\dots,{p_\ell}\} 
        = \big\{ P \in \mathscr{P}\{{p_1},\dots,{p_\ell}\} : \# B \le 2 \text{ for all } B \in P \big\}
    \end{align*}
    and $g_{\bf a}({\bf s})$ is defined by \eqref{E:g_a(s)}
    %, i.e.,
    %\begin{align*}
    %    g_{\bf a}({\bf s}) = -\Var\left( \sum_{i=1}^I a_i B(s_i) \right), \quad {\bf s} = (s_1, \dots, s_I).
    %\end{align*}
    with ${\bf a} = (a_1,\dots,a_I)$ being the non-trivial entries of $\boldsymbol{\alpha}:=\sigma_{k_m}\cdots\sigma_{k_{1}}\boldsymbol\varepsilon$ listed in the same order, which satisfy $\sum_{i=1}^Ia_i = 0$ and $|a_i|\le 3$ for all $1 \le i\le I$.
      \end{enumerate}
\end{proposition}

\begin{proof}
        (i).  By Lemma \ref{lem:sig:sig:G}, the $\sigma_i$'s commute with each other, so we can write
    \begin{align}\label{sigmaG:commute}
        \boldsymbol{\sigma}  G_{\boldsymbol\varepsilon}({\bf u}) = \partial_{u_{2k_1'-1}} \cdots \partial_{u_{2k_{\ell}}'-1} \sigma_{k_m}\cdots \sigma_{k_{1}} G_{\boldsymbol\varepsilon}({\bf u}).
    \end{align}
    By Lemma \ref{lem:var:reduc} (ii), we have $\sigma_{k_m}\cdots \sigma_{k_{1}} G_{\boldsymbol\varepsilon}({\bf u}) = G_{\sigma_{k_m}\cdots \sigma_{k_{1}} \boldsymbol{\varepsilon}}(\bf u)$, which is a function depending only on $u_{j_1},\cdots,  u_{j_I}$. This proves the first part. 

    (ii). Using Lemma \ref{lem:var:reduc} and the notations in Definition \ref{def_non-trivial-Variable}, we can write
    \begin{align*}
        \sigma_{k_m}\cdots \sigma_{k_1} G_{\boldsymbol\varepsilon}({\bf u}) = G_{\sigma_{k_m}\cdots \sigma_{k_{1}} \boldsymbol{\varepsilon}}({\bf u})=
        \exp(2\pi^2 g_{\bf a}({\bf s})).
    \end{align*}
    Thanks to \eqref{sigmaG:commute} and Lemma \ref{lem:fdb}, we have
    \begin{align*}
    %\label{E:fdb}
        \boldsymbol\sigma G_{\boldsymbol\varepsilon}({\bf s})
        =\partial_{s_{p_1}} \cdots \partial_{s_{p_{\ell}}} \exp(2\pi^2 g_{\bf a}({\bf s}))
        = \sum_{P\in \mathscr{P}\{p_1,\dots,p_{\ell}\}} (2\pi^2)^{\# P} \exp(2\pi^2 g_{\bf a}({\bf s})) \prod_{B \in P} g_{\bf a}^{(B)}({\bf s}).
    \end{align*}
    %where $\mathscr{P}\{p_1,\dots,p_{\ell}\}$ denotes the set of all partitions of $\{p_1, \dots,p_\ell\}$.
    But by Lemma \ref{lem:dg} (iii), $g_a^{(B)}({\bf s}) = 0$ whenever $\# B \ge 3$.
    Hence, the preceding formula simplifies to
    \begin{align*}
        \boldsymbol\sigma G_{\boldsymbol\varepsilon}({\bf s})
        = \sum_{P\in \mathscr{P}_2\{p_1,\dots,p_{\ell}\}} (2\pi^2)^{\# P} \exp(2\pi^2 g_{\bf a}({\bf s})) \prod_{B \in P} g_{\bf a}^{(B)}({\bf s}),
    \end{align*}
 which is our desired formula (\ref{E:fdb}). 
 The properties of $a_i$ follow from Lemma \ref{lem:var:reduc}.    
\end{proof}

\section{Variance estimates for fractional Brownian motion}\label{sec:variance}
In view of \eqref{E:fdb}, in order to estimate $\boldsymbol{\sigma}G_{\boldsymbol{\varepsilon}}({\bf s})$, we need to obtain estimates for the variance function \eqref{E:g_a(s)} and its first and second derivatives. 
These estimates will be provided in Propositions \ref{lem:LND}, \ref{lem:dg:bd} and \ref{lem:ddg:bd} below.
%we will provide the estimates for $g_a({\bf s})$, its first and second derivatives that are required for the main proofs of the theorem. 

\subsection{Estimate for the variance function}

In order to establish the estimate, we first recall some results about variance and conditional variance.

\begin{lemma}\cite[Lemma 8.1]{B73}\label{lem:Berman}
    Let $(X_1, \dots, X_n)$ be a Gaussian vector with mean 0, and let
    \begin{align*}
        Y_1 = X_1, \qquad Y_j = X_j-X_{j-1} \quad \text{for $j=2,\dots,n$.}
    \end{align*}
    Then, for all $a_1, \dots, a_n \in \R$,
    \begin{align*}
        \Var\left( \sum_{j=1}^n a_j Y_j \right)
        \ge \frac{\det\mathrm{Cov}(X_1, \dots, X_n)}{n \prod_{j=1}^n \sigma_j^2} \sum_{j=1}^n a_j^2 \sigma_j^2,
    \end{align*}
    where $\sigma_j^2 = \Var(Y_j)$ and $\mathrm{Cov}(X_1, \dots, X_n)$ is the covariance matrix of $(X_1,\dots,X_n)$.
\end{lemma}

\begin{lemma}\cite[Corollary A.2]{KM20}\label{lem:det:cov}
    For any centered Gaussian vector $(Z_1, \dots, Z_n)$,
    \begin{align*}
        &\det \mathrm{Cov}{(Z_1,\dots,Z_n)}
        = \Var(Z_1) \prod_{j=2}^n \Var(Z_j|Z_1,\dots,Z_{j-1}).
    \end{align*}
\end{lemma}

\begin{lemma}\cite[Lemma A.4]{KM20}\label{lem:CVar:proj}
    Let $\mathcal{H} \subset L^2(\Omega)$ be a Hilbert space with inner product $\langle X, Y \rangle_{L^2(\Omega)} = \E[XY]$ such that for any $n \in \N_+$, for any $X_1, \dots, X_n \in \mathcal{H}$, $\{X_i\}_{i=1}^n$ is a family of jointly Gaussian random variables with mean 0.
    Let $X\in\mathcal{H}$ and $\mathcal{G}$ be a closed linear subspace of $\mathcal{H}$.
    Then,
    \[
        \Var(X|\mathcal{G}) = \|X-P_{\mathcal{G}}(X)\|_{L^2(\Omega)}^2,
    \]
    where $P_{\mathcal{G}}(X)$ denotes the orthogonal projection of $X$ onto $\mathcal{G}$.
\end{lemma}

\begin{proposition}\label{lem:CVar:mono}
    Let $\mathcal{H} \subset L^2(\Omega)$ be a Hilbert space with respect to the inner product $\langle X, Y \rangle_{L^2(\Omega)} = \E[XY]$ such that for any $n \in \N_+$, for any $X_1, \dots, X_n \in \mathcal{H}$, $\{X_i\}_{i=1}^n$ is a family of jointly Gaussian random variables with mean 0.
    Let $X\in \mathcal{H}$ and $\mathcal{A}_1$, $\mathcal{A}_2$ be two subsets of $\mathcal{H}$.
    If $\mathcal{A}_1\subset \mathcal{A}_2$, then
    \begin{align*}
        \Var(X|\mathcal{A}_1) \ge \Var(X|\mathcal{A}_2).
    \end{align*}
\end{proposition}

\begin{proof}
    For $i=1,2$, let $\mathcal{G}_i$ be the closed linear subspace spanned by $\mathcal{A}_i$.
    By Lemma \ref{lem:CVar:proj} above,
    \[
    \Var(X|\mathcal{G}_i) = \|X-P_{\mathcal{G}_i}X\|_{L^2(\Omega)}^2,
    \]
    where $P_{\mathcal{G}_i}X$ denotes the orthogonal projection of $X$ onto the subspace $\mathcal{G}_i$.
    Note that $\mathcal{A}_i$ and $\mathcal{G}_i$ generate the same $\sigma$-algebra.
    It follows that
    \[
    \Var(X|\mathcal{A}_i) = \Var(X|\mathcal{G}_i) = \|X-P_{\mathcal{G}_i}X\|_{L^2(\Omega)}^2\quad \text{for $i=1,2$.}
    \]
    By the distance-minimizing property of orthogonal projections \cite[Theorem 2.6]{Conway} and by $\mathcal{G}_1 \subset \mathcal{G}_2$, we have
    \[
    \|X-P_{\mathcal{G}_1}X\|_{L^2(\Omega)}^2
    = \inf_{Y\in \mathcal{G}_1}\|X-Y\|_{L^2(\Omega)}^2 \ge \inf_{Y\in \mathcal{G}_2}\|X-Y\|_{L^2(\Omega)}^2 = \|X-P_{\mathcal{G}_2}X\|_{L^2(\Omega)}^2.
    \]
    Combining the last two displays completes the proof of the lemma.
\end{proof}

The lemma below is due to Pitt \cite{P78} and is known as the strong local nondeterminism property of fractional Brownian motion.

\begin{lemma}\cite[Lemma 7.1]{P78}\label{lem:Pitt}
    Let $\{B(t): t \ge 0\}$ be a fractional Brownian motion with Hurst parameter $H\in(0,1)$.
    Then, there exists a constant $C_H>0$ such that for any $t> 0$ and any $r\in (0, t]$,
    \begin{align*}
        \Var(B(t)|B(s):|s-t| \ge r) = C_H r^{2H}.
    \end{align*}
\end{lemma}

In the next lemma, we provide a key estimate for the variance function \eqref{E:g_a(s)}.

\begin{proposition}\label{lem:LND}
    Let $H\in(0,1)$.
    Then, there exists a constant $C_H>0$ depending only on $H$ such that for all integers $n \ge 1$, for all $0<t_1<\cdots<t_n$, for all $\alpha_1, \dots, \alpha_n \in \R$,
    \begin{align*}
        \Var\left( \sum_{j=1}^n \alpha_j B(t_j) \right) \ge \frac{C_H^{n-1}}{n} \sum_{j=1}^n (\alpha_j+\alpha_{j+1}+\dots+\alpha_n)^2 (t_j-t_{j-1})^{2H},
    \end{align*}
    where $t_0=0$.
\end{proposition}

\begin{proof}
    By rewriting the linear combination of $B(t_j)$ as a linear combination of increments (c.f. (\ref{E:increment}) replaced by $B(t)$), we see that it is equivalent to prove the existence of $C_H>0$ such that for all $n \ge 1$, for all $0<t_1<\dots<t_n$, for all $a_1, \dots, a_n\in \R$,
    \begin{align}\label{E:Var:LB}
        \Var\left( \sum_{j=1}^n a_j (B(t_j)-B(t_{j-1})) \right) \ge \frac{C_H^{n-1}}{n} \sum_{j=1}^n a_j^2 (t_j-t_{j-1})^{2H}.
    \end{align}
    In order to prove this, we first invoke Lemma \ref{lem:Berman} to find that
    \begin{align*}
        \Var\left( \sum_{j=1}^n a_j(B(t_j)-B(t_{j-1})) \right)
        \ge \frac{\det\mathrm{Cov}(B(t_1), \dots, B(t_n))}{n \prod_{j=1}^n (t_j-t_{j-1})^{2H}} \sum_{j=1}^n a_j^2 (t_j-t_{j-1})^{2H}.
    \end{align*}
    Then, by Lemma \ref{lem:det:cov} and Lemma \ref{lem:CVar:mono},
    \begin{align*}
        &\det\mathrm{Cov}(B(t_1), \dots, B(t_n))\\
        &= \Var(B(t_1)) \prod_{j=2}^n \Var(B(t_j)|B(t_1),\dots,B(t_{j-1}))\\
        &\ge \Var(B(t_1)) \prod_{j=2}^n \Var(B(t_j)|B(s): |s-t_j| \ge t_j-t_{j-1} ).
    \end{align*}
    Finally, we may appeal to Lemma \ref{lem:Pitt} to find a constant $C_H>0$ depending only on $H$ such that 
    \begin{align*}
        \Var(B(t_j)|B(s): |s-t_j| \ge t_j-t_{j-1} )=C_H (t_j-t_{j-1})^{2H}
    \end{align*}
    for all $n \in \N_+$, $0<t_1<\dots<t_n$ and $2\le j \le n$.
    This proves \eqref{E:Var:LB} and completes the proof of Lemma \ref{lem:LND}.
\end{proof}

\subsection{Estimates for the partial derivatives}

The following two propositions provide the key estimates for the first and second derivatives of $g_{\bf a}({\bf s})$ defined in \eqref{E:g_a(s)}.

\begin{proposition}\label{lem:dg:bd}
    Suppose $H\in(1/2,1)$.
    Let $I\in\N_+$ and ${\bf a} = (a_1,\dots,a_I) \in (\Z \setminus\{0\})^I$ satisfy $\sum_{i=1}^I a_i = 0$.
    Let $d\in \N_+$ and $1=i_1<i_2<\cdots<i_d < I$ be such that
    \begin{align}\label{E:a:cond}
        a_i + a_{i+1} + \cdots + a_I  \begin{cases}
            = 0 & \text{for } i\in \{i_1,\dots,i_d\},\\
            \ne 0 & \text{for }i\in \{1,\dots, I\} \setminus\{i_1,\dots,i_d\}.
        \end{cases}
    \end{align}
    Set $i_{d+1}=I+1$ and $\mathcal{J}^* = \{1,\dots, I\} \setminus\{i_1,\dots,i_d\}$.
    Then, for any $0<s_1<\cdots<s_I$ and $1\le k \le I$,
    \begin{align}\label{E:dg:bd}
        |\partial_{s_k} g_{\bf a}({\bf s})|
        \le 4H I \|{\bf a}\|_\infty^2 \sum_{i\in \mathcal{J}^*} (s_i-s_{i-1})^{2H-1}.
    \end{align}
\end{proposition}

\begin{proof}
    Let $j_k$ be the unique $j$ such that $i_{j-1} \le k < i_j$. Then, by Lemma \ref{lem:dg} (i), we can write
    \begin{align*}
        &\partial_{s_k} g_{\bf a}({\bf s})
        = 2H a_k \left[ \sum_{1\le i<k} a_i (s_k-s_i)^{2H-1} - \sum_{ k+1\le i\le I} a_i (s_i-s_k)^{2H-1} \right]\\
        &= 2H a_k \left[ \sum_{j=2}^{j_k-1}\sum_{ i_{j-1}\le i<i_j} a_i (s_k-s_i)^{2H-1} + \sum_{ i_{j_k-1} \le i \le k-1} a_i(s_k-s_i)^{2H-1} \right.\\
        & \quad \left. - \sum_{ k+1 \le i \le i_{j_k}} a_i(s_i-s_k)^{2H-1} -\sum_{j=j_k+1}^{d+1}\sum_{ i_{j-1}\le i < i_j} a_i (s_i-s_k)^{2H-1} \right].
    \end{align*}
    It follows from \eqref{E:a:cond} that, for each $2\le j \le d+1$,
    \[
        a_{i_{j-1}}+a_{i_{j-1}+1}+\cdots+a_{i_j-1} = 0
    \]
    since $ a_{i_{j}}+\cdots+ a_I=0$ and $a_{i_{j-1}}+\cdots+a_I=0$. Hence,
    \[
         a_{i_{j-1}} = -(a_{i_{j-1}+1}+\cdots+a_{i_j-1}).
    \]
    Then we can use the preceding to further write
    \begin{align}\label{eq_5.7}
        \partial_{s_k} g_{\bf a}({\bf s})
        = 2H a_k &\left[ \sum_{j=2}^{j_k-1}\sum_{ i_{j-1}< i<i_j} a_i \Big((s_k-s_i)^{2H-1}-(s_k-s_{i_{j-1}})^{2H-1}\Big) \right.\nonumber\\
        & + \sum_{i_{j_k-1} \le i \le k-1} a_i (s_k-s_i)^{2H-1} - \sum_{k+1\le i < i_{j_k}} a_i (s_i-s_k)^{2H-1}\\
        & \left. - \sum_{j=j_k}^{d+1}\sum_{ i_{j-1}<i < i_j} a_i \Big((s_i-s_k)^{2H-1} - (s_{i_{j-1}}-s_k)^{2H-1} \Big) \right]\nonumber.
    \end{align}
    Since $2H-1\in(0,1)$, we may use the elementary inequality $(x+y)^p \le x^p + y^p$ for $p\in(0,1)$ and $x, y \ge 0$ to deduce that
    \begin{align*}
        &(s_k-s_i)^{2H-1} - (s_k-s_{i_{j-1}})^{2H-1} 
        \le (s_i-s_{i_{j-1}})^{2H-1} \text{ for $1\le j\le j_k-1$, $i_{j-1}<i<i_j$,}\\
        &(s_i-s_k)^{2H-1}-(s_{i_{j-1}}-s_k)^{2H-1} \le (s_i-s_{i_{j-1}})^{2H-1} \text{ for $j_k<j\le d+1$, $i_{j-1}<i<i_j$,}
    \end{align*}
    and
    \begin{align*}
        &(s_k-s_i)^{2H-1} \le \sum_{i < \ell \le k} (s_\ell-s_{\ell-1})^{2H-1} \le \sum_{i_{j_k-1}< \ell < i_{j_k}}(s_\ell-s_{\ell-1})^{2H-1} \text{ for $i_{j_k-1} \le i < k$,}\\
        &(s_i-s_k)^{2H-1} \le \sum_{k < \ell \le i} (s_\ell-s_{\ell-1})^{2H-1} \le \sum_{i_{j_k-1}< \ell < i_{j_k}}(s_\ell-s_{\ell-1})^{2H-1} \text{ for $k< i < i_{j_k}$.}
    \end{align*}
    It follows that
    \begin{align*}
        &|\partial_{s_k} g_{\bf a}({\bf s})|\\
        &\le 2H \|{\bf a}\|_\infty^2 \left[\sum_{j=2}^{d+1} \sum_{i_{j-1}<i<i_j} (s_i-s_{i_{j-1}})^{2H-1} +\sum_{i_{j_k-1}\le i < i_{j_k}} \sum_{i_{j_k-1}<\ell<i_{j_k}} (s_{\ell}-s_{\ell-1})^{2H-1} \right]\\
        &\le 4H \|{\bf a}\|_\infty^2 \sum_{j=2}^{d+1} \sum_{i_{j-1}<i<i_j} \sum_{i_{j-1}<\ell< i_j} (s_\ell-s_{\ell-1})^{2H-1}\\
        &\le 4H I \|{\bf a}\|_\infty^2 \sum_{j=2}^{d+1} \sum_{i_{j-1}<\ell<i_j} (s_\ell-s_{\ell-1})^{2H-1}.
    \end{align*}
%    This proves \eqref{E:dg:bd}.
%    For $1\le k < l \le I$,
%    \begin{align*}
%        \partial_{s_k} \partial_{s_l} g_a({\bf s}) = - 2H (2H-1) a_k a_l (s_l-s_k)^{2H-2}.
%    \end{align*}
%    Since the $a_i$'s are all non-zero, $i_j-i_{j-1}\ge 2$ for all $j\in\{1,\dots,m\}$.
%    This implies that there exists $i\in\{1,\dots,I\}\setminus\{i_1, \dots, i_m\}$ such that $s_k \le s_{i-1} < s_i \le s_l$, and hence by monotonicity,
%    \begin{align*}
%        |\partial_{s_k} \partial_{s_l} g_a({\bf s})|
%        \le 2H (2H-1) \|a\|_\infty^2 (s_l-s_k)^{2H-2} 
%        \le 2H (2H-1) \|a\|_\infty^2 (s_i-s_{i-1})^{2H-2}.
%    \end{align*}
%    This shows \eqref{E:ddg:bd} and completes the proof of Lemma \ref{lem:dg:bd}.
In the first inequality, the first term is obtained from combining the estimates of the first and last term of (\ref{eq_5.7}) and the second term comes from estimating the two middle terms of (\ref{eq_5.7}). The second inequality comes from concatenating all the consecutive differences.   Note that this inequality is exactly our desired estimate, so the  proof of Lemma \ref{lem:dg:bd} is complete.
\end{proof}

\begin{proposition}\label{lem:ddg:bd}
       Suppose that $H\in(1/2,1)$.
    Let ${\bf a} \in (\Z \setminus\{0\})^I$ be such that $\sum_{i=1}^I a_i = 0$. 
    Let $i_1 < \cdots < i_d$ and $\mathcal{J}^* = \{1,\dots, I\}\setminus\{i_1, \dots, i_d\}$ be as defined in Lemma \ref{lem:dg:bd},
    and $g_{\bf a}({\bf s}) = g_{\bf a}(s_1, \dots, s_I)$ as defined in \eqref{E:g_a(s)}.
    Let $\{p_1<p_2<\cdots<p_{2n}\}$ be a subset of $\{1,\dots,I\}$ such that $p_{i+1} - p_i \ge 2$ for all $1\le i < n$.  Suppose that  $\{\{p_{j_1}, p_{j_2}\}, \dots, \{p_{j_{2n-1}}, p_{j_{2n}}\}\}$ is a partition of $\{p_1,p_2,\dots,p_{2n}\}$.
    %with $p_{j_{2k-1}} < p_{j_{2k}}$ for all $1\le k\le n$ and $p_{j_1} < p_{j_3} < \cdots < p_{j_{2n-1}}$.

    Then, there exists a set of indices $\tau(1)<\tau(2)<\cdots < \tau(n)$ such that 
    %$p_{j_{2k-1}}\le \tau(k)\le p_{j_{2k-1}+1}$ and 
    $\tau(k)\in{\mathcal J}^{\ast}$ for all $1\le k\le n$ and 
    %it satisfies the following: 
    \begin{align*}
        \prod_{k=1}^n |\partial_{s_{p_{j_{2k-1}}}} \partial_{s_{p_{j_{2k}}}} g_{\bf a}({\bf s})| \le 2^n \|{\bf a}\|_\infty^{2n} \prod_{k=1}^n (s_{\tau(k)}-s_{\tau(k)-1})^{2H-2}.
        %\\\le 2^n \|a\|_\infty^{2n} \sum_{K\subset \mathcal{J}^*: \# K = n}\prod_{k\in K} \left[ 1 + (s_k-s_{k-1})^{2H-2}\right].
    \end{align*}
\end{proposition}

\begin{proof}
    Without loss of generality, we may and will order the partition elements such that $p_{j_{2k-1}} < p_{j_{2k}}$ for all $1\le k\le n$ and $p_{j_1} < p_{j_3} < \cdots < p_{j_{2n-1}}$.
    By Lemma \ref{lem:dg} (ii) and $H\in (1/2,1)$,
    \begin{align}\label{E:prod:ddg:bd}
        \prod_{k=1}^n |\partial_{s_{p_{j_{2k-1}}}} \partial_{s_{p_{j_{2k}}}} g_{\bf a}({\bf s})|
        \le 2^n \|{\bf a}\|_\infty^{2n} \prod_{k=1}^n (s_{p_{j_{2k}}}-s_{p_{j_{2k-1}}})^{2H-2}.
    \end{align}
    For each $1 \le k \le n$, by monotonicity and $2H-2\in(-1,0)$,
    \begin{align*}
        (s_{p_{j_{2k}}}-s_{p_{j_{2k-1}}})^{2H-2}
        \le (s_{p_{j_{2k-1}+1}} - s_{p_{j_{2k-1}}})^{2H-2}.
    \end{align*}
    By the assumption, ${p_{j_{2k-1}+1}} - {p_{j_{2k-1}}} \ge 2$ for each $1 \le k \le n$, and 
    \begin{align}\label{E:sum:a_i}
        a_i + a_{i+1} + \dots + a_I 
        \begin{cases}
        = 0 & \text{if } i\in \{i_0, \dots, i_d\},\\
        \ne 0 & \text{if } i \in \mathcal{J}^*.
        \end{cases}
    \end{align}
    The key observation is that for each $1 \le k \le n$, there exists an index $\tau(k) \in \{1,\dots,I\}\setminus\{ i_1,\dots,i_d\}$ such that $p_{j_{2k-1}} < \tau(k) \le p_{j_{2k-1}+1}$.
    In fact, if this is not true, then $\{p_{j_{2k-1}} + 1, \dots, p_{j_{2k-1}+1}\} \subset \{i_0, \dots, i_d\}$, which implies that at least two indices in $\{i_1, \dots, i_d\}$ are consecutive numbers, say $i_r$ and $i_{r+1} = i_r+1$, but then \eqref{E:sum:a_i} would imply that $a_{i_r} = 0$, which is a contradiction.
    Therefore, for each $1 \le k \le n$, we can find an index $\tau(k) \in {\mathcal J}^{\ast}$ such that $p_{j_{2k-1}} < \tau(k) \le p_{j_{2k-1}+1}$.
    Then, by monotonicity again,
    \begin{align*}
        (s_{p_{j_{2k-1}+1}} - s_{p_{j_{2k-1}}})^{2H-2} \le (s_{\tau(k)}-s_{\tau(k)-1})^{2H-2}.
    \end{align*}
    Also, because of the ordering of the $p_{j_k}$'s, we have 
    \begin{align*}
        p_{j_1} < p_{j_1+1} \le p_{j_3} < p_{j_3 +1} \le p_{j_5} < \cdots \le p_{j_{2n-3}} < p_{j_{2n-3}+1} \le p_{j_{2n-1}},
    \end{align*}
    which implies that the $\tau(k)$'s are all non-repeating.
    This together with \eqref{E:prod:ddg:bd} shows that there is a set of indices $\tau(1)< \tau(2) < \cdots < \tau(n)$ in $\mathcal{J}^*$ such that
    \begin{align*}
        \prod_{k=1}^n |\partial_{s_{p_{j_{2k-1}}}} \partial_{s_{p_{j_{2k}}}} g_{\bf a}({\bf s})|
        &\le 2^n \|{\bf a}\|_\infty^{2n} \prod_{k=1}^n (s_{\tau(k)}-s_{\tau(k)-1})^{2H-2}.
        %& \le 2^n \|a\|_\infty^{2n} \sum_{K \subset \mathcal{J}^*: \# K = n}\prod_{k\in K} \left[ 1+ (s_{k}-s_{k-1})^{2H-2} \right].
    \end{align*}
    This completes the proof of Proposition \ref{lem:ddg:bd}.
\end{proof}

\section{Proof of Theorem \ref{theorem_hard-estimate-introd}}
\label{S:Fourier:decay}
 We are now ready to prove Theorem \ref{theorem_hard-estimate-introd}, which will complete the proof of our main theorem. 
 %Let us restate it below for convenience. 
%\begin{theorem} (=Theorem \ref{theorem_hard-estimate-introd})\label{theorem_hard-estimate}
%    Suppose $H\in(1/2,1)$.
%    Then, there exists a constant $C \in (1,\infty)$ depending only on $H$ such that for all $q\in\N_+$, for all $\boldsymbol\sigma = (\sigma_1,\dots,\sigma_q) \in \Omega^q$, for all $\boldsymbol{\varepsilon} \in \mathcal{A}_{2q}$, for all $T>0$,
%    \begin{align*}
%        |\mathcal{I}[\boldsymbol\sigma, G_{\boldsymbol\varepsilon}]| \le
%        C^{q^2} T^q.
%    \end{align*}
%\end{theorem}

\begin{proof}[Proof of Theorem \ref{theorem_hard-estimate-introd}, Part (1)]
    Fix $\boldsymbol{\varepsilon}\in \mathcal{A}_{2q}$ and $\boldsymbol\sigma = (\sigma_1,\dots,\sigma_q)\in \Omega^q$.
    Define the indices $k_1<\cdots<k_m$ and $k_1'<\cdots<k_{\ell}'$ such that $m+\ell=q$,
    \[
    \{k_1,\dots,k_{m}\} = \{k: \sigma_k = \Phi_{2k-1}^+ \text{ or } \Phi_{2k-1}^-\},
    \]
    and
    \[
    \{k'_1,\dots,k'_{\ell}\}=\{k:\sigma_k=\partial_{u_{2k-1}}\}.
    \]
    Recall from \eqref{eq_I[e,G]} and \eqref{eq_I_sigma} that
    \begin{align*}
        \mathcal{I}[\boldsymbol{\sigma}, G_{\boldsymbol{\varepsilon}}] = \mathcal{I}_T^\lambda[\boldsymbol{\sigma}, G_{\boldsymbol{\varepsilon}}]
        &= (-1)^{\# J_2(\boldsymbol{\sigma})} \cdot \mathcal{I}[\sigma_{k_m} \cdots \sigma_{k_1}\boldsymbol{\varepsilon}\,,
        \boldsymbol{\sigma} G_{\boldsymbol{\varepsilon}}]\\
        &= (-1)^{\# J_2(\boldsymbol{\sigma})} \int_{\Delta_{\sigma_{k_m} \cdots \sigma_{k_1}\boldsymbol{\varepsilon},T}} 
        e^{-2\pi i \lambda \langle \sigma_{k_m} \cdots \sigma_{k_1}\boldsymbol{\varepsilon}, {\bf u} \rangle}
        \boldsymbol{\sigma}   G_{\boldsymbol{\varepsilon}}({\bf u}) \, d{\bf u}.
    \end{align*} 
    By Proposition \ref{Lemma_deBruno} (i), $\boldsymbol{\sigma}   G_{\boldsymbol{\varepsilon}}({\bf u}) = \boldsymbol{\sigma}   G_{\boldsymbol{\varepsilon}}({\bf s})$ depends only on the non-trivial variables ${\bf s} = (s_1,\dots, s_I)$ (see Definition \ref{def_non-trivial-Variable}).
    Note that $\sigma_{k_m} \cdots \sigma_{k_1}\boldsymbol{\varepsilon}$ has $m$ entries being equal to $\ast$ and $\mathcal{I}[\boldsymbol{\sigma}, G_{\boldsymbol{\varepsilon}}]$ is a $2q-m$ iterated integral (see Lemma \ref{lemma_properties}). 
    Hence, there are $2q-m-I$ many trivial variables and
    \begin{equation}\label{E:I:ds}
        \left|\mathcal{I}[\boldsymbol{\sigma}, G_{\boldsymbol{\varepsilon}}] \right|\le  T^{2q-m-I}\int_{0<s_1<\cdots <s_I<T}  |\boldsymbol\sigma G_{\boldsymbol\varepsilon}({\bf s})| d{\bf s},
    \end{equation}
    where $T^{2q-m-I}$ comes from integrating the trivial variables corresponding to the non-$\ast$ entries.
    Thanks to Fa\`a di Bruno's formula (Proposition \ref{Lemma_deBruno} (ii)), we can write
    \begin{equation}\label{E:FdB1}
        \boldsymbol\sigma G_{\boldsymbol\varepsilon}({\bf s})
        = \sum_{P\in \mathscr{P}_2\{{p_1},\dots,{p_{\ell}}\}} (2\pi^2)^{\# P} \exp(2\pi^2 g_{\bf a}({\bf s})) \prod_{B \in P} g_{\bf a}^{(B)}({\bf s}),
    \end{equation}
    where ${\bf a} = (a_1,\dots,a_I)$ denotes the non-trivial entries in $\sigma_{k_m}\cdots\sigma_{k_1}\boldsymbol{\varepsilon}$,
    \[
        g_{\bf a}({\bf s}) = -\Var\left( \sum_{i=1}^I a_i B(s_i)\right),
    \]
     and ${p_1},\dots,{p_{\ell}}$ are the indices for the differentiation variables $s_{{p_1}},\cdots,  s_{{p_{\ell}}}$ which are parts of the non-trivial variables $s_1, \dots, s_I$ (see Remark \ref{rmk:entry} (i)). These variables and indices are determined by $\boldsymbol{\varepsilon}$ and $\boldsymbol{\sigma}$.

    Next, we estimate \eqref{E:FdB1} term by term. Fix a partition $P\in \mathscr{P}_2\{p_1,\dots, p_\ell\}$, 
    and let 
    \begin{align}\begin{split}\label{E:n:r}
        &n = n(P) := \#\{ B \in P : \# B = 2 \},\\
        &r = r(P) := \#\{ B \in P : \# B = 1\} = \ell-2n.
    \end{split}\end{align}
    Recall from Lemma \ref{lem:var:reduc} (iii) that $\sum_{i=1}^Ia_i=0$.
    Let $1=i_1<\cdots<i_d<I$ be such that
    \[
        a_i+a_{i+1}+\dots+a_I 
        \begin{cases}
            =0 & \text{for } i\in \{i_1,\dots,i_d\},\\
            \ne 0 & \text{for } i \in \mathcal{J}^*,
        \end{cases}
    \]
    where
    \[
        \mathcal{J}^*:=\{1,\dots, I\} \setminus\{i_1,\dots, i_d\}.
    \]
    The value of $d$ and the set $\mathcal{J}^*$ depend on $\boldsymbol{\varepsilon}$ and $\boldsymbol{\sigma}$. We have the following estimates:
    \begin{enumerate}
        \item {\bf Variance estimate}:  By Proposition \ref{lem:LND} and $I \le 2q$, we find that for some constant $1<K<\infty$ depending only on $H$,
    \begin{align}\begin{split}\label{E:exp:var:bd}
        \exp(2\pi^2 g_{\bf a}({\bf s}))
        %&\le \exp\left\{ -9\pi I C_H^{I-1} \sum_{i\in \mathcal{J}^*} (s_i-s_{i-1})^{2H} \right\}
        &\le \exp\left\{ -K^{-q} \sum_{i\in \mathcal{J}^*} (s_i-s_{i-1})^{2H} \right\}.
    \end{split}\end{align}
    \smallskip
    \item {\bf First derivative bound}: By Proposition \ref{lem:dg:bd}, the relation $r=\ell-2n$, and the bound $\|a\|_{\infty}\le 3$ (see Proposition \ref{Lemma_deBruno} (ii)), we have
    \begin{align*}
        \prod_{B \in P : \# B = 1} |g^{(B)}_{\bf a}({\bf s})| \le \left[ 36 I  \sum_{j\in \mathcal{J}^*} (s_j-s_{j-1})^{2H-1} \right]^{\ell-2n}.
    \end{align*}
    Hence, by the elementary inequality $(\sum_{i=1}^N |a_i|)^M \le N^M \sum_{i=1}^N |a_i|^M$ for all integers $M, N\ge0$ as well as $I \le 2q$ and $\ell-2n\le \ell\le q$,
    \begin{align}\label{E:prod:dg:bd8}
        \prod_{B \in P : \# B = 1} |g^{(B)}_{\bf a}({\bf s})| 
        &\le (144)^q q^{2q} \sum_{j\in \mathcal{J}^*} (s_j-s_{j-1})^{(2H-1)(\ell-2n)}.
        %& \le 36^q q^{2q} \sum_{i \in \mathcal{J}^*} \left[ 1 + (s_i-s_{i-1})^{2H-1} \right]^q.
    \end{align}
    \smallskip
    \item  {\bf Second derivative bound}: Suppose that
    \begin{align*}
        \bigcup_{B \in P: \# B = 2} B = \{p_1, p_2, \dots, p_{2n}\},
    \end{align*}
    where $p_1 < p_2 < \cdots < p_{2n}$ and 
    \[
    \{B \in P : \# B = 2\} = \{ \{p_{j_1},p_{j_2}\}, \dots, \{p_{j_{2n-1}}, p_{j_{2n}} \}\}
    \]
    with $p_{j_{2k-1}}<p_{j_{2k}}$ for all $1\le k \le n$.
    By Lemma \ref{lem:dvar:sep}, $p_{i+1}-p_i \ge 2$ for all $1\le i \le n$.
    Hence, we may appeal to Proposition \ref{lem:ddg:bd} and the bound $\|a\|_{\infty}\le 3$ to find that
    \begin{align}\begin{split}\label{E:prod:ddg:bd2}
        \prod_{B \in P : \# B = 2} |g_{\bf a}^{(B)}({\bf s})|
        & \le 2^n \|{\bf a}\|_\infty^{2n} \prod_{k=1}^n (s_{\tau(k)}-s_{\tau(k)-1})^{2H-2}\\
        \le& (18)^q \prod_{k=1}^n (s_{\tau(k)}-s_{\tau(k)-1})^{2H-2}\\
    %    & \le 18^q \sum_{K \subset \mathcal{J}^*: \# K = \lfloor \ell/2 \rfloor} \prod_{k\in K} \left[ 1+ (s_{k}-s_{k-1})^{2H-2} \right],
    \end{split}\end{align}
    where we have used $n \le q$ to obtain the last inequality and $\tau(1),\dots, \tau(n)\in{\mathcal J}^{\ast}$ are obtained from Proposition \ref{lem:ddg:bd}.
    \smallskip
    \item {\bf Number of possible partitions}: 
    %Let $P_{2,n}(\ell)$ denote the total number of partitions $P \in \mathscr{P}_2\{1,\dots, \ell\}$ such that $\#\{ B \in P : \# B = 2\} = n$.
%    Moreover, note that the last expression in \eqref{E:sig:G:bd} depends only on the variables $\{ s_1, \dots, s_I \} \setminus \{s_{i_1},\dots,s_{i_d}\}$.
%    Let us relabel these variables by $\tilde{s}_1 < \cdots < \tilde{s}_{I-d}$ and write $\tilde{\bf s} = (\tilde{s}_1,\dots, \tilde{s}_{I-d})$.
    %For each $n\in \{0,\dots,\lfloor \ell/2\rfloor\}$, 
    %
    %     Also, $\mathscr{P}_2\{{p_1},\dots,{p_{\ell}}\}$ can be written as the disjoint union
%    \[
%         \mathscr{P}_2: = \mathscr{P}_2\{{p_1},\dots,{p_{\ell}}\}
%        = \bigcup_{n=0}^{\lfloor \ell/2\rfloor} \mathscr{P}_{2,n}\{{p_1},\dots,{p_{\ell}}\},
%    \]
%    where
%    \[
%        \mathscr{P}_{2,n}\{{p_1},\dots,{p_{\ell}}\}
%        = \big\{ P \in \mathscr{P}_2\{{p_1},\dots,{p_{\ell}}\} : \# \{B \in P : \# B = 2\} = n  \big\}.
%    \]
    Denote 
    \[
    \mathscr{P}_2 = \mathscr{P}_2(\boldsymbol{\varepsilon},\boldsymbol{\sigma}) := \mathscr{P}_2\{p_1,\dots,p_\ell\}.
    \]
    See Lemma \ref{lem:fdb} and Proposition \ref{Lemma_deBruno} for the definitions of $\mathscr{P}\{\cdots\}$ and $\mathscr{P}_2\{\cdots\}$.
    Clearly, $\mathscr{P}_2$ can be written as the disjoint union
    \[
        \mathscr{P}_2
        = \bigcup_{n=0}^{\lfloor \ell/2\rfloor} \left\{ P \in \mathscr{P}_2\{{p_1},\dots,{p_{\ell}}\} : \begin{array}{l}
            \# \{B \in P : \# B = 2\} = n \text{ and}\\
            \# \{ B \in P : \# B = 1\} = \ell-2n
        \end{array}  \right\}.
    \]
    In particular, we see that every partition $P \in \mathscr{P}_2\{p_1,\dots,p_\ell\}$ can be split into two disjoint partitions, i.e., $P = Q \cup R$, where $\# B = 2$ for all $B \in Q$ and $\# B = 1$ for all $R$.
    Observe that $P$ is completely determined once $Q$ is fixed (because $R$ is simply formed by single elements in the complement).
    This observation leads to the following identity and crude estimate:
    \begin{align*}
        \qquad \mathscr{P}_2 = \bigcup_{n=0}^{\lfloor \ell/2\rfloor} \bigcup_{\substack{A \subset \{p_1,\dots,p_\ell\}\\ \# A = 2n}} \left\{Q \cup R : \begin{array}{l}
        Q \in \mathscr{P}(A) \text{ with } \# B = 2,\, \forall B \in Q,\\
        \text{and } R = \{ \{a\} : a \in \{p_1,\dots,p_\ell\}\setminus A \}
        \end{array}\right\}
    \end{align*}    
    \begin{align}\begin{split}\label{E:P:bd}
        \qquad\#\mathscr{P}_2
        &= \sum_{n=0}^{\lfloor \ell/2\rfloor} \sum_{\substack{A \subset \{p_1,\dots,p_\ell\}\\ \# A = 2n}} \#\big\{ Q \in \mathscr{P}(A) : \forall B \in Q, \# B = 2  \big\}\\
        &\le \sum_{n=0}^{\lfloor \ell/2\rfloor} \sum_{\substack{A \subset \{p_1,\dots,p_\ell\}\\ \# A = 2n}} (2n)^{2n} = \sum_{n=0}^{\lfloor \ell/2\rfloor} \binom{\ell}{2n} (2n)^{2n}
        \le \ell \cdot \ell^\ell \cdot \ell^\ell
        \le q^{2q+1}.
    \end{split}\end{align}
    One can derive a tighter bound for $\# \mathscr{P}_2$ but we shall not do that here since the crude estimate above is enough for our purpose.
    \smallskip
    \item {\bf Number of integrating variables}: Finally, we observe that the quantity $a_{p_i+1}+\cdots+a_I$ must be non-zero
    %at the differentiation variable $s_{p_i}=u_{2k'_i-1}$ 
    for every $i \in \{1,\dots, \ell\}$
    (which corresponds to a differentiation variable $s_{p_i}=u_{2k'_i-1}$) because this quantity would be equal to  $\varepsilon_{2k_i'} + \cdots + \varepsilon_{2q}$, which is non-zero since it is a sum of an odd number of $\varepsilon_j \in \{-1,1\}$.
    This implies that
    \begin{align}\label{E:I-d<m}
        I-d \ge \ell.
    \end{align}
    \end{enumerate}

   Applying the estimates (\ref{E:exp:var:bd}), (\ref{E:prod:dg:bd8}) and (\ref{E:prod:ddg:bd2}) to (\ref{E:FdB1}), we obtain
   $$
    \begin{aligned}    
    &\,|\boldsymbol{\sigma} G_{\boldsymbol{\varepsilon}}({\bf s})|\\
      &\le C_0^q q^{2q} \sum_{P\in \mathscr{P}_2} e^{-K^{-q}\sum_{i\in{\mathcal J}^{\ast}}(s_i-s_{i-1})^{2H}}\sum_{j \in \mathcal{J}^*} (s_j-s_{j-1})^{(2H-1)(\ell-2n)} \prod_{k=1}^n(s_{\tau(k)}-s_{\tau(k)-1})^{2H-2}\\
    &=C_0^qq^{2q} \sum_{P\in \mathscr{P}_2}
   \sum_{j \in \mathcal{J}^*}\left\{ (s_j-s_{j-1})^{(2H-1)(\ell-2n)} \prod_{k\in{\mathcal J}^{\ast}\setminus\{\tau(1),\dots,\tau(n)\}} e^{-K^{-q}(s_{k}-s_{k-1})^{2H}} \right.\\
    &  \hskip1.6in \left.\times \prod_{k=1}^n\left[(s_{\tau(k)}-s_{\tau(k)-1})^{2H-2}\, e^{-K^{-q}(s_{\tau(k)}-s_{\tau(k)-1})^{2H}}\right]\right\},
    \end{aligned}
    $$
    where $s_0=0$ and $C_0 = 144\cdot18\cdot2\pi^2$. 
    % Note that the last expression depends only on the variables $\{ s_1, \dots, s_I \} \setminus \{s_{i_1},\dots,s_{i_d}\}$.
    Note that the last expression depends only on $\{ s_i-s_{i-1} \}_{i \in {\mathcal J}^{\ast}}$.
    Let us relabel the variables $\{ s_i-s_{i-1} \}_{i \in {\mathcal J}^{\ast}}$ by $\{\tilde{s}_i\}_{i=1,2,\dots,I-d}$ and write $\tilde{\bf s} = (\tilde{s}_1,\dots, \tilde{s}_{I-d})$. 
    Define $\Lambda_{P,j}(\tilde{s}_1,\dots,\tilde{s}_{I-d})$
    as the expression inside $\{ \cdots\}$ in the last display, so that 
    $$
   |\boldsymbol{\sigma} G_{\boldsymbol{\varepsilon}}({\bf s})|\le C_0^qq^{2q}\sum_{ P\in\mathscr{P}_2}\sum_{j\in{\mathcal J}^{\ast}}\Lambda_{P,j}(\tilde{s}_1,\dots,\tilde{s}_{I-d}).
    $$
   Putting this back into \eqref{E:I:ds}, changing variables to $\tilde{\bf s}$, and then enlarging the domain of integration of each variable $\tilde{s}_i$ uniformly to $[0,T]$ yield the following:
   \begin{align}\begin{split}\label{eq_I_estimate-purple}
       &\left| \mathcal{I}[\boldsymbol\sigma, G_{\boldsymbol\varepsilon}] \right| 
        \le T^{2q-m-I} \int_{0<s_1<\cdots<s_I<T} |\boldsymbol{\sigma} G_{\boldsymbol{\varepsilon}}({\bf s})| d{\bf s}\\
       &\quad \le C_0^qq^{2q}T^{2q-m-I+d}\sum_{P\in \mathscr{P}_2}\sum_{j\in{\mathcal J}^{\ast}}\int_{[0,T]^{I-d}} \Lambda_{P,j}(\tilde{s}_1,\dots,\tilde{s}_{I-d})\,d{\tilde{\bf s}}.
       %d\tilde{s}_1\cdots d\tilde{s}_{I-d}. 
        \end{split}\end{align}
        Note that each $\Lambda_{P,j}$ is a non-negative function of the form
        $$
        \Lambda_{P,j}(\tilde{s}_1,\dots,\tilde{s}_{I-d})= \prod_{i=1}^{I-d}\psi_i(\tilde{s}_i)
        $$
        where $\psi_i\ge 0$ for each $i$, so we can decompose the iterated integral and write
       $$
       \begin{aligned}
{\mathcal I}_{P,j} :=\int_{[0,T]^{I-d}} \Lambda_{P,j}(\tilde{s}_1,\dots,\tilde{s}_{I-d})\,d{\tilde{\bf s}} 
% = &\int_0^T\int_{\tilde{s}_1}^T\cdots
% %\int_{\tilde{s}_{I-d-2}}^T
% \int_{\tilde{s}_{I-d-1}}^T \left(\prod_{i=1}^{I-d}\psi_i(\tilde{s}_i-\tilde{s}_{i-1}) \right)d\tilde{s}_{I-d}
% %d\tilde{s}_{I-d-1}
% \cdots d\tilde{s}_{2}\, d\tilde{s}_{1} \\
% \le &\int_0^T\int_{0}^{T}\cdots
% %\int_{0}^{T}
% \int_{0}^{T} \left(\prod_{i=1}^{I-d}\psi_i(u_i)\right)du_{I-d}\cdots du_2 \, du_1\\
=& \prod_{i=1}^{I-d} \left(\int_0^T\psi_i(u)du\right).
        \end{aligned}$$
 Hence, according to the explicit forms of the functions $\psi_i$'s, we see that
      \begin{align}\label{eq-IPj}
  &\,{\mathcal I}_{P,j} = \int_{[0,T]^{I-d}} \Lambda_{P,j}(\tilde{\bf s})\,d{\tilde{\bf s}}\\
  %d\tilde{s}_1\cdots d\tilde{s}_{I-d}\\
  &=
  \begin{cases}\left(\int_0^Tu^{2H-2}e^{-\frac{u^{2H}}{K^q}}du\right)^n \left(\int_0^T e^{-\frac{u^{2H}}{K^q}}du\right)^{I-d-n-1}  \left(\int_0^T u^{(2H-1)(\ell-2n)}e^{-\frac{u^{2H}}{K^q}}du \right) ~ \mbox{or}\medskip\\
  \left(\int_0^Tu^{2H-2}e^{-\frac{u^{2H}}{K^q}}du\right)^{n-1} \left(\int_0^T e^{-\frac{u^{2H}}{K^q}}du\right)^{I-d-n}  \left(\int_0^T u^{(2H-1)(\ell-2n)}u^{2H-2}e^{-\frac{u^{2H}}{K^q}}du \right) \nonumber
           \end{cases}
       \end{align}
      where the first case happens when $j\not\in\{\tau(1),\cdots, \tau(n)\}$ and the second case happens when $j\in\{\tau(1),\cdots, \tau(n)\}$.

      We now estimate each integral in \eqref{eq-IPj}. By the change of variable $u = K^{q/(2H)} x$, we see that
    \begin{gather*}
    \int_0^{\infty}u^{2H-2}e^{-K^{-q}u^{2H}}du = K^{q(1-\frac{1}{2H})}\int_0^{\infty} x^{2H-2}e^{-x^{2H}}dx\le K_1^q,\\
    \int_0^{\infty} e^{-K^{-q}u^{2H}}du = K^{\frac{q}{2H}}\int_0^\infty e^{-x^{2H}}dx \le K_1^q
    \end{gather*}
    for some constant $0<K_1<\infty$ depending only on $H$. Note that the integrals are all finite because $1/2 < H < 1$.  For the last two integrals in \eqref{eq-IPj}, we have
    $$
    \begin{aligned}
        \int_0^{\infty} u^{(2H-1)(\ell-2n)}e^{-K^{-q}u^{2H}}du= &~K^{\frac{q}{2H}(1+(2H-1)(\ell-2n))}\int_0^{\infty} x^{(2H-1)(\ell-2n)} e^{-x^{2H}}dx\\
        \le &~K^{\frac{q}{2H}(1+(2H-1)q)}\left(1+\int_1^{\infty} x^{(2H-1)q} e^{-x^{2H}}dx\right),
    \end{aligned}
    $$
    where we used $0\le \ell-2n\le q$ to deduce the last inequality. Similarly, 
    $$ 
    \begin{aligned}
    \int_0^{\infty} u^{(2H-1)(\ell-2n)}u^{2H-2}e^{-K^{-q}u^{2H}}du  &\le K^{\frac{q}{2H}(1+(2H-1)q+2H-2)} \\
    &\times\left(\int_0^1x^{2H-2}dx+\int_1^{\infty} x^{(2H-1)q+2H-2} e^{-x^{2H}}dx\right).
   \end{aligned}
    $$
   By the change of variable $x^{2H} = v^2/2$ and standard Gaussian moment estimates, it is not hard to see that
     \begin{align*}
        \int_0^\infty x^{(2H-1)q}e^{-x^{2H}} du \le C_1^q q^q \quad \text{and} \quad \int_0^{\infty} x^{(2H-1)q+2H-2}e^{-x^{2H}}dx \le C_1^qq^q,
     \end{align*}
     where $0<C_1<\infty$ is a constant depending only on $H$. 
     Hence, we deduce that
     \begin{gather*}
        \int_0^{\infty} u^{(2H-1)(\ell-2n)}e^{-K^{-q}u^{2H}}du \le K_2^{q^2} q^q,\\
    \int_0^{\infty} u^{(2H-1)(\ell-2n)}u^{2H-2}e^{-K^{-q}u^{2H}}du  \le K_2^{q^2} q^q
    \end{gather*}
    for some constant $0<K_2<\infty$ depending only on $H$. Finally, we divide the proof into two cases.

   \noindent  {\bf Case 1:} $T<1$. We bound the integrals in \eqref{eq-IPj} involving $u^{2H-2}$ by integrating to infinity and keep the other integrals without $u^{2H-2}$. In both cases, the integrands without $u^{2H-2}$ are less than 1, so we have 
    \begin{align*}
    {\mathcal I}_{P,j}&\le \begin{cases} K_1^{qn}\cdot  T^{I-d-n-1}\cdot T & \text{if } j \not\in\{\tau(1),\dots, \tau(n)\}\\
  K_1^{q(n-1)} \cdot T^{I-d-n} \cdot K_2^{q^2}q^q& \text{if } j \in \{\tau(1),\dots,\tau(n)\}
           \end{cases}\\
           &\le C_2^{q^2} q^q T^{I-d-n}
    \end{align*}
    for some constant $0<C_2<\infty$ depending only on $H$. Putting this back into \eqref{eq_I_estimate-purple} and using the bound $\#{\mathscr P}_2\le q^{2q+1}$ from \eqref{E:P:bd}, $\#{\mathcal J}^{\ast}\le 2q$,  $n\le \ell/2$ and $q=\ell+m$, we get that
    $$
    \begin{aligned}
    \left| \mathcal{I}[\boldsymbol\sigma, G_{\boldsymbol\varepsilon}] \right|&\le C_0^q q^{2q} T^{2q-m-I+d} \sum_{P\in \mathscr{P}_2}\sum_{j\in{\mathcal J}^{\ast}} C_2^{q^2} q^q T^{I-d-\ell/2} \\
    &\le  C^{q^2} T^{q+\ell-\frac{\ell}{2}} = {C}^{q^2} T^{q+\frac{\ell}{2}} \\
    &\le C^{q^2}  T^{q} 
    \end{aligned}
    $$
    for some sufficiently large constant $C>1$ depending on $C_0$ and $C_2$, which both depend only on $H$ (noting that $q^{cq} = e^{cq\ln q}\le e^{cq^2}$). In the last inequality, we used the fact that $T<1$.

    \noindent   {\bf Case 2:} $T\ge 1$. In this case, replacing $T$ by $\infty$ for all integrals in (\ref{eq-IPj}) and using $I-d \le 2q$, we see that there is a constant $C_2>1$ depending only on $H$ such that ${\mathcal I}_{P,j} \le C_2^{q^2}$ for all $P\in \mathscr{P}_2$ and $j \in \mathcal{J}^*$. Hence, going back to \eqref{eq_I_estimate-purple} and counting all partitions, we can find a large constant $C$, depending only $H$ such that 
       $$
       \left| \mathcal{I}[\boldsymbol\sigma, G_{\boldsymbol\varepsilon}] \right|\le C^{q^2} T^{2q-m-I+d} = C^{q^2} T^{q+\ell-(I-d)}.
       $$
       But $I-d\ge \ell$ thanks to (\ref{E:I-d<m}). Since $T\ge 1$, the above is less than $C^{q^2}T^{q}$.

       Combining both cases yields \eqref{E:|I|:bd} with a constant $C>1$ depending only on $H$ but not on any of the parameters $q, \boldsymbol{\sigma}, \boldsymbol{\varepsilon}, \lambda$ or $T$.
       This completes the proof of Part (1) of Theorem \ref{theorem_hard-estimate-introd}. 
\end{proof}

Next, we establish a less restrictive (yet more technical) version of Lemma \ref{lem:by:parts}:
\begin{lemma}\label{lem:by:parts:2}
    Let $\boldsymbol{\varepsilon} \in (\Z \cup \{\ast\})^n$, $j \in \{1,\dots,n\}$.
    Suppose that $\varepsilon_j \not\in \{0, \ast\}$ and $\varepsilon_{j-1},\varepsilon_{j+1} \ne \ast$.
    Let $F$ be a function that is well defined and continuously differentiable on $\overline{\Delta_{\boldsymbol{\varepsilon},T}}$ except possibly on the diagonals $\{u_k = u_\ell\}$ where $1 \le k < \ell \le j-2$.
    Suppose $\Phi_j^\pm(F) \in L^1(\Delta_{\Phi_j^\pm(\boldsymbol{\varepsilon}),T})$ and $\partial_{u_j} F \in L^1(\Delta_{\boldsymbol{\varepsilon},T})$.
    Then
    \[
        {\mathcal I}[\boldsymbol\varepsilon, F] = \frac{1}{2\pi i \lambda\varepsilon_j}\left({\mathcal I} [\Phi_j^-(\boldsymbol\varepsilon), \Phi_j^-(F)]-{\mathcal I} [\Phi_j^+(\boldsymbol\varepsilon),\Phi_j^+(F)]+ {\mathcal I}[\boldsymbol\varepsilon, {\partial_{u_j} F}]\right).
    \]
\end{lemma}
\begin{proof}
    The assumptions imply that for a.e.~$0<u_1<\cdots<u_{j-1}<u_{j+1}<\cdots<u_n<T$, the function $u_j \mapsto F({\bf u})$ is absolutely continuous on $[u_{j-1},u_{j+1}]$, so we may perform integration by parts (see, e.g., \cite[Theorem 3.36]{folland}) and evaluate $\Phi_j^\pm(F)$ at $u_j = u_{j\pm 1}$ as in the proof of Lemma \ref{lem:by:parts}.
    The integrability conditions allow to use Fubini's theorem to combine iterated integrals and obtain the result.
\end{proof}

\begin{proof}[Proof of Theorem \ref{theorem_hard-estimate-introd}, Part (2)]
Recall the notation \eqref{E:Sigma^r}--\eqref{eqJ_123}.
Set $\sigma_{k_i'}(\boldsymbol{\varepsilon}) = \boldsymbol{\varepsilon}$ so that we may write 
$\sigma_{k_m}\cdots\sigma_{k_1}(\boldsymbol{\varepsilon}) = \sigma_r \cdots \sigma_1(\boldsymbol{\varepsilon})$.
We have shown in Part (1) of Theorem \ref{theorem_hard-estimate-introd} that 
for any $(\sigma_1,\dots, \sigma_q) \in \Omega^q$,
\begin{align}\label{integrable:q}
    \sigma_q \cdots \sigma_1 G \in L^1(\Delta_{\sigma_q\cdots\sigma_1(\boldsymbol{\varepsilon}),T}).
\end{align}
Next, we show that for any $r \in \{1,\dots, q-1\}$ and $\boldsymbol{\sigma} = (\sigma_1,\dots, \sigma_r) \in \Omega^r$,
\begin{align}\label{integrable:r}
    \sigma_r \cdots \sigma_1 G \in L^1(\Delta_{\sigma_r\cdots\sigma_1(\boldsymbol{\varepsilon}),T}).
\end{align}
In fact, the fundamental theorem of calculus yields
\begin{align*}
	\sigma_r \cdots \sigma_1 G({\bf u})
	&= \Phi_{2q-1}^- \cdots \Phi_{2r+1}^- \sigma_r \cdots \sigma_1 G({\bf u})\\
	&\quad + \int_{u_{2q-2}}^{u_{2q-1}}  \cdots \int_{u_{2r}}^{u_{2r+1}} \,\partial_{u_{2q-1}}\cdots \partial_{u_{2r+1}}\sigma_r\cdots \sigma_1 G(\tilde{\bf u})~d \tilde{u}_{2r+1}...d\tilde{u}_{2q-1},
\end{align*}
where $\tilde{\bf u}$ is obtained from ${\bf u}$ by replacing $u_{2r+1},u_{2r+2},\dots,u_{2q-1}$ by $\tilde{u}_{2r+1},\tilde{u}_{2r+2},\dots,\tilde{u}_{2q-1}$ respectively.
Write 
\begin{align*}
	&\boldsymbol{\sigma}^- = (\sigma_1, \dots, \sigma_r, \Phi_{2r+1}^-,\dots,\Phi_{2q-1}^-),\\
	&\boldsymbol{\sigma}' = (\sigma_1, \dots, \sigma_r, \partial_{u_{2r+1}}, \dots, \partial_{u_{2q-1}}),
\end{align*}
which are both elements of $\Omega^q$.
It follows that
\begin{align*}
&\int_{\Delta_{\sigma_r\cdots\sigma_1(\boldsymbol{\varepsilon}), T}} , |\sigma_r \cdots \sigma_1 G({\bf u})| ~d {\bf u}\\\
	&\le \int_{\Delta_{\boldsymbol{\sigma}(\boldsymbol{\varepsilon}), T}} \, |\boldsymbol{\sigma}^- G({\bf u})| ~d {\bf u} +\int_{\Delta_{\boldsymbol{\sigma}(\boldsymbol{\varepsilon}), T}} \int_{u_{2q-2}}^{u_{2q-1}}  \cdots \int_{u_{2r}}^{u_{2r+1}} \, |\boldsymbol{\sigma}' G(\tilde{\bf u})|~d\tilde{u}_{2r+1}\cdots d \tilde{u}_{2q-1}d{\bf u} \\
	& \le T^{q-r} \int_{\Delta_{\boldsymbol{\sigma}^-(\boldsymbol{\varepsilon}), T}} \, | \boldsymbol{\sigma}^- G({\bf u})|~d {\bf u}
	+ T^{q-r} \int_{\Delta_{\boldsymbol{\sigma}'(\boldsymbol{\varepsilon}),T}} \, |\boldsymbol{\sigma}'G(\tilde{\bf u})|~d \tilde{\bf u} < \infty,
\end{align*}
where the last two integrals are finite thanks to \eqref{integrable:q}.
This shows \eqref{integrable:r}.

Finally, we show \eqref{fBm:IBP} by induction.
First, we may apply Lemma \ref{lem:by:parts:2} to $F=G_{\boldsymbol{\varepsilon}}$ with $n=2q$ and $j=1$ since $G_{\boldsymbol{\varepsilon}}$ is defined and differentiable on $\overline{\Delta_{\boldsymbol{\varepsilon},T}}$ and the integrability conditions follows from \eqref{integrable:r}.
This yields
\[
    \mathcal{I}[\boldsymbol{\varepsilon}, G_{\boldsymbol{\varepsilon}}]
    = \frac{1}{2\pi i\lambda} \cdot \frac{1}{\varepsilon_1} \sum_{\boldsymbol{\sigma} \in \Omega^1} \mathcal{I}[\boldsymbol{\sigma}, G_{\boldsymbol{\varepsilon}}].
\]
Suppose that the following holds for some $r \in \{1,\dots, q-1\}$:
\[
    \mathcal{I}[\boldsymbol{\varepsilon}, G_{\boldsymbol{\varepsilon}}]
    = \left(\frac{1}{2\pi i\lambda}\right)^r \cdot \prod_{i=1}^r \frac{1}{\varepsilon_{2i-1}} \sum_{\boldsymbol{\sigma} \in \Omega^r} \mathcal{I}[\boldsymbol{\sigma}, G_{\boldsymbol{\varepsilon}}].
\]
Recall that $\mathcal{I}[\boldsymbol{\sigma}, G_{\boldsymbol{\varepsilon}}]$ in the above expression is given by
\begin{align*}
    \mathcal{I}[\boldsymbol{\sigma}, G_{\boldsymbol{\varepsilon}}]
    = (-1)^{\# J_2(\boldsymbol{\sigma})} \, \mathcal{I}[\sigma_r \cdots \sigma_1(\boldsymbol{\varepsilon}), \sigma_r \cdots \sigma_1(G_{\boldsymbol{\varepsilon}})].
\end{align*}
Set $j = 2r+1$ and $F=\sigma_r \cdots \sigma_1(G_{\boldsymbol{\varepsilon}})$. By Lemma \ref{lem:dg} and Proposition \ref{Lemma_deBruno}, any possible singularities of $F$ must appear on the diagonals $\{u_k = u_\ell\}$ where $1 \le k<\ell \le j-2$ are a pair of odd indices.
Also, $\Phi_j^\pm(F)$ and $\partial_{u_j}F$ are integrable thanks to \eqref{integrable:r}. Therefore, we may apply Lemma \ref{lem:by:parts:2} to this $F$ and complete the induction step as in Proposition \ref{prop_by_part-iterate}.
\end{proof}

\noindent{\bf Acknowledgments.} The project was initiated when both authors were visiting the Chinese University of Hong Kong (CUHK) in Fall 2024. As alumni of CUHK, they would like express their sincere gratitudes to Professor De-Jun Feng for his invitations and hospitality as well as many useful discussions. Chun-Kit Lai was partially supported by the AMS-Simons Research Enhancement Grants for Primarily Undergraduate Institution (PUI) Faculty.
Cheuk Yin Lee was partially supported by a research startup fund of the Chinese University of Hong Kong, Shenzhen and the Shenzhen Peacock fund 2025TC0013. Finally, the authors would like to thank the referees for their careful reading of this paper and many valuable comments.

\bibliographystyle{amsalpha}
\bibliography{biblio.bib}

\end{document}